\documentclass{article}

\usepackage{amsmath}
\usepackage{amsthm}
\usepackage{amssymb}
\usepackage{mathtools}
\usepackage{mathrsfs}
\usepackage{nicefrac}

\usepackage{graphicx}
\graphicspath{ {./} }
\usepackage{float}
\usepackage{caption}
\usepackage{subcaption}

\usepackage[colorlinks=true, allcolors=blue]{hyperref}
\usepackage[shortlabels]{enumitem}
\usepackage{xcolor}
\usepackage{verbatim}
\usepackage{siunitx}

\theoremstyle{plain}
\newtheorem{theorem}{Theorem}[section]

\theoremstyle{plain}
\newtheorem{theoremalph}{Theorem}

\theoremstyle{plain}
\newtheorem*{theoremb}{Theorem B}

\theoremstyle{plain}
\newtheorem{proposition}[theorem]{Proposition}

\theoremstyle{plain}

\theoremstyle{plain}

\theoremstyle{definition}
\newtheorem{definition}[theorem]{Definition}

\theoremstyle{definition}

\theoremstyle{definition}

\theoremstyle{definition}

\theoremstyle{definition}

\theoremstyle{plain}
\newtheorem{lemma}[theorem]{Lemma}

\theoremstyle{remark}
\newtheorem{remark}[theorem]{Remark}

\theoremstyle{plain}
\newtheorem{corollary}[theorem]{Corollary}

\newcommand{\set}[1]{\{#1\}}

\newcommand{\m}[1]{\mathbb{#1}}
\newcommand{\mrm}[1]{\mathrm{#1}}

\newcommand{\inv}{^{-1}}
\newcommand{\toinfty}{^{\infty}}
\newcommand{\cinfclosure}[1]{\overline{#1}\toinfty}

\newcommand{\conj}[2]{{#1}{#2}{#1}\inv}

\newcommand{\rtwo}{\m R^2}
\newcommand{\ztwo}{\m Z^2}
\newcommand{\nstar}{\m N^{*}}
\newcommand{\sone}{\m S^1}
\newcommand{\diam}{\mathrm{diam}}
\newcommand{\den}{\mathrm{den}}

\newcommand{\torus}{\m{T}^2}
\newcommand{\obartor}{\cinfclosure{\mathcal{O}}(\torus)}
\newcommand{\obartormes}{\cinfclosure{\mathcal{O}_{\mu}}(\torus)}
\newcommand{\finegraph}[1]{C^{\dagger} (#1)}
\newcommand{\homeo}{\mrm{Homeo}}

\newcommand{\idtor}{\mrm{Id}_{\torus}}

\newcommand{\subc}[1]{_{C^{#1}}}
\newcommand{\cdist}[3]{d\subc{#1}(#2,#3)}

\newcommand{\shapes}{\mathcal{H}}
\newcommand{\la}[2]{\mrm{LA}(#1,#2)}
\newcommand{\isla}[3]{#1 \in \la{#2}{#3}}
\newcommand{\opnorm}[1]{\|#1\|}
\newcommand{\zon}{\mrm{Zon}}
\newcommand{\stre}{\mrm{Str}}

\title{Torus diffeomorphisms with parabolic and non-proper actions on the fine curve graph and their generalized rotation sets}
\author{Nastaran Einabadi}

\begin{document}
\maketitle

\begin{abstract}
We prove that a generic element of the Anosov-Katok class of the torus, $\obartor$, acts parabolically and non-properly on the fine curve graph $\finegraph{\torus}$. Additionally, we show that a generic element of $\obartor$ admits generalized rotation sets of any point-symmetric compact convex homothety type in the plane.
\end{abstract}

\section{Introduction}
The fine curve graph of a surface $S$, denoted $\finegraph{S}$, was introduced by Bowden, Hensel, and Webb in \cite{bowden2022quasi}. It is analogous to the classical curve graph $C(S)$ which has been extensively studied and used to explore the mapping class group $\mrm{MCG}(S)$. The fine curve graph is a Gromov hyperbolic metric space \cite{bowden2022quasi}. Unlike in the case of the classical curve graph, curves are not considered up to isotopy in $\finegraph{S}$. Hence the fine curve graph is a tool to study the homeomorphism or diffeomorphism groups of a surface. For example, in \cite{bowden2022quasi} the authors showed that for a closed orientable surface $S$ of positive genus, $\mathrm{Diff}_0(S)$ admits unbounded quasi-morphisms. As a result, it is not uniformly perfect and its fragmentation norm is unbounded, answering a question of Burago, Ivanov, and Polterovich \cite{burago2008conjugation}. The definition of the fine curve graph is as follows:
\begin{definition}
    Let $S_g$ be a closed orientable surface of genus $g \geq 1$. The vertices of the \textit{fine curve graph} $\finegraph{S_g}$ are non-contractible simple closed curves on $S_g$. In the case $g \geq 2$, there is an edge between two vertices if the corresponding curves are disjoint. If $g=1$, there is an edge between vertices intersecting at most once. 
\end{definition}
In this article, we provide examples of torus diffeomorphisms acting parabolically and non-properly on $\finegraph{\torus}$ (see Definitions \ref{def:hyp-par-ell} and \ref{def:proper}). The orbits under the action of such diffeomorphisms on $\finegraph \torus$ are unbounded but return to a bounded region infinitely many times. This provides us with a natural example of a non-proper parabolic isometry of a Gromov hyperbolic space. Such isometries have previously been constructed by extending the Edelstein examples from Hilbert spaces to infinite-dimensional hyperbolic spaces via the Poincaré extension \cite{edelstein1964on}. More explanation can be found in an answer by Yves de Cornulier on math overflow \cite{320408} and the book of Das, Simmons, and Urba\'{n}ski which puts an emphasis on non-proper hyperbolic geometry \cite[ch.~11.1.2]{Das2017geometry}. In general, examples of this type do not occur in most usual hyperbolic spaces like proper hyperbolic spaces. We find these diffeomorphisms in the Anosov-Katok class of the torus, denoted by $\obartor$ (this is the set of $C\toinfty$ limits of smooth conjugates of rigid rotations, see equation \eqref{eq:obartor} for a definition).
\begin{theoremalph}
    \label{thm:parabolic-nonproper}
    A generic diffeomorphism in $\obartor$ (similarly in $\obartormes$) acts parabolically and non-properly on the fine curve graph.
\end{theoremalph}
Additionally, these diffeomorphisms have a bizarre sublinear rotational behaviour. \emph{Generalized rotation sets} were introduced in \cite{guiheneuf2023parabolic} (see Definition \ref{def:sublin-rot-set}). They capture the sublinear diffusion rate of orbits of torus diffeomorphisms. The following theorem shows the abundance of possible homothety types of generalized rotation sets for these maps (see Definiton \ref{def:homothety-type}).
\begin{theoremalph}
\label{thm:generic-all-symmetric}
    A generic diffeomorphism in $\obartor$ (similarly in $\obartormes$) admits generalized rotation sets of all point-symmetric compact convex homothety types in $\rtwo$.
\end{theoremalph}
\begin{remark}
    As of now, we are not aware of any argument implying that non-symmetric generalized rotation sets do not occur in this class. The construction of such examples can be the subject of further investigation.
\end{remark}
\subsection{Context and methods}
The group of homeomorphisms of a surface $\homeo(S)$ acts on $\finegraph S$ by graph automorphisms and hence by isometries. In fact, it was shown in \cite{long2021automorphisms} that for a closed orientable surface $S$ of genus $\geq 2$, the automorphism group of $\finegraph S$ is isomorphic to the group of homeomorphisms of $S$. This result was then extended to the case of the torus and more general surfaces by Le Roux and Wolff \cite{roux2022automorphisms}, where they consider the non-seperating transverse curve graph. \par
Isometries of $\finegraph S$ can be classified as hyperbolic, parabolic, and elliptic (see Definition \ref{def:hyp-par-ell}). It is possible to determine the class of the action of a torus homeomorphism by looking at its dynamical, and more specifically rotational behaviour. This is due to \cite{bowden2022rotation} and \cite{guiheneuf2023parabolic}. For higher genus surfaces the hyperbolic case was done by Guihéneuf and Militon using the ergodic homological rotation set \cite{guiheneuf2023hyperbolic}.\par
Let $f: \torus \to \torus$ be a homeomorphism homotopic to the identity, and $\Tilde{f}: \rtwo \to \rtwo$ a lift of $f$. Let $D$ be a fundamental domain of the action of $\ztwo$ on $\rtwo$ by translations. Then the rotation set of $\Tilde{f}$ can be defined as:
\begin{equation*}
    \rho(\Tilde{f}) = \lim_{n \to \infty} \frac{\Tilde{f}^n(D)}{n}
\end{equation*}
where the limit is taken in the Hausdorff topology. This limit exists, and is independent of the choice of the fundamental domain $D$ by the work of Misiurewicz and Ziemian \cite{misiurewicz1989rotation}. The rotation set $\rho(\Tilde{f})$ is a compact convex subset of $\rtwo$ \cite{misiurewicz1989rotation}, which captures the linear speed of the orbits. Up to integer translations, the set $\rho(\Tilde{f})$ is independent of the choice of the lift $\Tilde{f}$. \par
We still do not have a complete picture of which convex subsets of $\rtwo$ can be realized as a rotation set. In the case of empty interior, it's possible to find examples whose rotation sets are singletons, rational segments passing through rational points or irrational segments with a rational extremity \cite{handel1989periodic}. Franks and Misiurewicz conjectured that these are the only possibilities \cite{franks1990rotation}. Le Calvez and Tal proved in \cite{lecalvez2018forcing} that the boundary of a rotation set cannot contain a segment of irrational slope which contains a rational point in its interior. In particular, no irrational segment having a rational point in its interior can be realized as a rotation set. However Avila has shown the existence of a diffeomorphism whose rotation set is a segment of irrational slope not passing through any rational points, providing a counter-example to the Franks-Misiurewicz conjecture. The case of rational segments is still open. In \cite{kwapisz1992polygon} Kwapisz showed that any rational polygon can be realized as a rotation set. But these are not the only possible rotation sets with non-empty interior. In his later work \cite{kwapisz1995nonpolygonal}, he provided an example of a rotation set with infinitely many rational extreme points accumulating on two partially irrational extreme points. It was shown in \cite{boyland2016new} that totally irrational extreme points are also possible. \par
In \cite{guiheneuf2023parabolic}, \emph{generalized rotation sets} were introduced. They were then used to give a criterion for non-ellipticity of the action of a homeomorphism on the fine curve graph. While the classical rotation set does not capture sublinear diffusion of orbits, such growth can be captured by the generalized rotation set. Fix $D = [0,1]^2$, and $\Tilde{x}_0 \in D$.
\begin{definition}
\label{def:sublin-rot-set}
    If for some sequence of natural numbers $(n_i)_{i \in \m N}$ the limit
    \begin{equation*}
        \lim_{i \to \infty} \frac{\Tilde{f}^{n_i}(D) - \Tilde{f}^{n_i}(\Tilde{x}_0)}{\diam(\Tilde{f}^{n_i}(D))}
    \end{equation*}
    exists in the Hausdorff topology, and $\lim_{i \to \infty} \diam(\Tilde{f}^{n_i}(D)) = +\infty$, then this limit is called a \emph{generalized rotation set}.
\end{definition}
It was shown in \cite{guiheneuf2023parabolic} that generalized rotation sets are compact convex subsets of $\rtwo$ with diameter $1$ containing the origin. \par
We use the approximation by conjugation method of Anosov and Katok \cite{anosov1970new} to find our exotic examples. This method has been a useful tool for providing examples of diffeomorphisms with interesting dynamical properties like minimality, unique ergodicity, and weak-mixing. The scheme can be applied to any manifold admitting a locally free smooth circle action. The maps obtained by this method fall in the Anosov-Katok class of diffeomorphisms, that is, the $C\toinfty$-closure of smooth conjugates of elements of a locally free smooth circle action. Later, Fathi and Herman \cite{fathi1977existence} combined this method with generic arguments to prove the genericity of such properties in this class.\par
On the torus, denote by $R_{\theta}$ the rigid translation sending $x$ to $x+\theta$ for some vector $\theta \in \torus$. Then, the Anosov-Katok class can be written as follows:
\begin{equation}
\label{eq:obartor}
    \obartor = \cinfclosure{\{\conj{h}{R_{\theta}} \mid h \in \mathrm{Diff}\toinfty(\torus),\theta \in \torus \}}.
\end{equation}
Consider this class with the usual $C\toinfty$ metric, with respect to which it is a complete metric space. Note that any smooth faithful circle action on $\torus$ is free \cite[thm.~9.3]{bredon1972introduction}, and that any free circle action on the torus is smoothly conjugate to a circle action by translations in the horizontal direction. Therefore the Anosov-Katok class is the same no matter which circle action is chosen.
\begin{definition}
    Let $X$ be a topological space, and let $S \subseteq X$ be a subset. We say $S$ is \emph{residual} if it is a countable intersection of open dense subsets. A property is said to be \emph{(Baire) generic} in $X$ if it holds on a residual set.
\end{definition}
Note that by definition, a countable intersection of residual sets is residual. Hence, if we have countably many generic properties in a space, a generic element in that space has all of those properties.
\begin{remark}
    Let $\obartormes$ denote the area preserving Anosov-Katok class, for which the conjugation maps $h$ are chosen to be area preserving diffeomorphisms. All of our proofs work for both $\obartor$ and $\obartormes$. In the rest of the article, we will only write down the arguments for $\obartor$.
\end{remark}
\begin{remark}
    By a result of Katok, a $C\toinfty$ diffeomorphism of a compact surface with positive topological entropy admits a hyperbolic horseshoe. Then by structural stability of hyperbolic sets, any $C\toinfty$ perturbation of such a diffeomorphism still admits a horseshoe and has positive entropy. So the set of torus diffeomorphisms with positive entropy is an open subset of $\mrm{Diff}\toinfty(\torus)$. As rotations have zero topological entropy, all elements of the Anosov-Katok class $\obartor$ have zero topological entropy as well.
\end{remark}
\subsection{Outline}
In \cite{kocsard2009mixing}, Kocsard and Koropecki introduced weak spreading torus homeomorphisms (see Definition \ref{def:weak-spreading}), and used the methods of Fathi and Herman to prove that they are generic in $\obartor$. In Section \ref{sec:non-proper-parabolic}, we combine this result with the genericity of smoothly rigid maps in $\obartor$ and the classification of isometries by Guihéneuf and Militon \cite{guiheneuf2023parabolic} to prove Theorem \ref{thm:parabolic-nonproper}. \par
The possible homothety types of generalized rotations sets of a generic element of $\obartor$ are investigated in Section \ref{sec:generic-shapes}. We will prove Theorem \ref{thm:generic-all-symmetric} by adapting Kocsard and Koropecki's approach in \cite{kocsard2009mixing} to using the Fathi-Herman method. This strategy is laid out in Subsection \ref{sec:fathi-herman}. For a summary of the ideas appearing in the proof, see the beginning of Section \ref{sec:generic-shapes}. The rest of the section is dedicated to the introduction of necessary preliminaries and to the details of the proofs.
\subsection*{Acknowledgements}
I would like to express my deepest gratitude to Pierre-Antoine Guihéneuf, whose insightful discussions and ideas have had a profound impact on this work. I also wish to thank Yves de Cornulier, Sebastian Hensel, Yusen Long, Emmanuel Militon, and Maxime Wolff for their useful comments, corrections, and stimulating discussions related to the subject of this paper.

\section{Non-proper parabolic isometries of the fine curve graph}
\label{sec:non-proper-parabolic}
We plan to prove Theorem \ref{thm:parabolic-nonproper}. Let us clarify the statement of the theorem. Let $\varphi$ be an isometry of any metric space $(X,d)$. For $x \in X$, the limit
$$\lim_{n \to \infty} \frac{d({\varphi}^n(x),x)}{n}$$
exists and does not depend on $x$. We call this value the \emph{asymptotic translation length} of $\varphi$, and denote it by $|\varphi|_X$.
\begin{definition}
\label{def:hyp-par-ell}
    Let $\varphi$ be an isometry of $(X,d)$. It is said to be
    \begin{enumerate}[a)]
        \item \textit{Hyperbolic} if $|\varphi|_X$ is positive.
        \item \textit{Parabolic} if $|\varphi|_X = 0$ and no orbit of $\varphi$ in $X$ is bounded.
        \item \textit{Elliptic} if $|\varphi|_X = 0$ and there exists a bounded orbit for $\varphi$ in $X$ (in which case all of the orbits will be bounded).
    \end{enumerate}
\end{definition}
\begin{definition}
\label{def:proper}
    An isometry $\varphi$ of $(X,d)$ is said to be \textit{proper} if for any point $x \in X$ and any infinite set of natural numbers $I$, the set $\{ \varphi^i(x) \}_{i \in I} \subseteq X$ is unbounded.
\end{definition}

\begin{remark}
    This means that for any point $x$, the inverse image of a bounded subset of $X$ under the function $i \mapsto \varphi^i(x)$ is bounded, which justifies the use of the terminology \emph{proper} in analogy with continuous maps under which the inverse image of a compact set is compact. In the literature the term metrically proper is also used and it is equivalent to our definition of proper for the action of $\m Z$ by iterates of the isometry.
\end{remark}

The concept of rigidity was adapted to topological dynamics by Glasner and Maon \cite{glasner1989rigidity}. Let us repeat a smooth version of the definition here.

\begin{definition}
\label{def:rigidity}
    Let $f$ be a $C\toinfty$ diffeomorphism of the torus $\torus$. It is said to be \emph{$C\toinfty$-rigid} if there exists an increasing sequence of positive integers $(n_i)_{i \in \m N}$ such that
    \begin{equation*}
        \lim_{i \to + \infty} \cdist{\infty}{f^{n_i}}{\idtor} = 0.
    \end{equation*}
\end{definition}
The following result is classic. See \cite[thm.~4.9]{yancey2013weakly} for a similar result in the non-smooth case.
\begin{proposition}
    \label{prop:gen-rig}
    The set of $C\toinfty$-rigid diffeomorphisms is generic in $\obartor$ and $\obartormes$.
\end{proposition}

\begin{proof}
    For $q \in \m N$, denote by $\Phi_n: \obartor \to \obartor$ the map sending $f$ to $f^n$. This is a continuous map of $\obartor$ with respect to the $C\toinfty$ topology\footnote{This can be seen using the multivariate Faà di Bruno formula \cite{Constantine1996AMF}.}. For each $k \in \nstar$, let the set $P_k \subseteq \obartor$ be as follows:
    \begin{equation}
    \label{eq:Pk}
        P_k = \displaystyle\bigcup_{n \geq k}
        \Phi_n\inv\left(B\subc{\infty}(\idtor,\frac{1}{k})\right).
    \end{equation}
    By continuity of the maps $\Phi_n$ this is an open subset of $\obartor$. Let us define $Q \subseteq \obartor$ as:
    \begin{equation}
    \label{eq:Q}
        Q=\left\{\conj{h}{R_{\frac{p}{q}}} \mid p \in \mathbb{Z}^2, q \in \nstar, h \in \mathrm{Diff}\toinfty (\torus) \right\}.
    \end{equation}
    For each $k \in \nstar$, we have $(\conj{h}{R_{\frac{p}{q}}})^{kq} = \idtor$, so $Q \subseteq P_k$. On the other hand, $Q$ is dense in $\obartor$, therefore $P_k$ is also dense. Hence, $P=\cap_k P_k$ is a Baire generic subset of $\obartor$. Let $f$ be any map in $P$. For each $k \in \nstar$ there exists $n_k \geq k$ such that $\cdist{\infty}{f^{n_k}}{\idtor} < \frac{1}{k}$. So, $f$ is $C\toinfty$-rigid. Therefore $P$ is a residual set of $C\toinfty$-rigid maps in $\obartor$.
\end{proof}

\begin{definition}
\label{def:epsilon-dense}
    Let $X,Y \subseteq \rtwo$, and $\varepsilon > 0$. The set $X$ is said to be \emph{$\varepsilon$-dense} in $Y$ if $Y \subseteq {B}_{\varepsilon}(X)$.
\end{definition}

Weak-spreading torus homeomorphisms were introduced as follows in \cite{kocsard2009mixing}. This property is stronger than topological weak mixing.
\begin{definition}
\label{def:weak-spreading}
    A homeomorphism $f$ of the torus $\torus$ is said to be \textit{weak spreading}, if it has a lift $\Tilde{f}$ to $\mathbb{R}^2$ such that for any non-empty open set $U \subseteq \mathbb{R}^2$, any $\varepsilon > 0$, and any $R > 0$, there exist an integer $n > 0$, and a ball of radius $R$ in which $\Tilde{f}^n(U)$ is $\varepsilon$-dense.
\end{definition}
\begin{theorem}[\cite{kocsard2009mixing}]
    \label{thm:weakspreading}
    Weak spreading diffeomorphisms are generic in $\obartor$ and $\obartormes$.
\end{theorem}
\begin{definition}
    Let $v \in \m R^2 \setminus \{0\}$. We say $f \in \mathrm{Homeo}(\torus)$ has \emph{bounded deviation in direction} $v$, if there exist $M>0$, $\rho \in \rtwo$, and a lift $\Tilde{f}:\m R^2 \to \m R^2$ of $f$ such that for any $x \in \m R^2$ and $n \in \m N$, $| \langle \Tilde{f}^n(x) - x - n \rho, v \rangle | < M$.
\end{definition}
\begin{remark}
    The reader may easily check that weak spreading maps have unbounded deviation in all directions, since they spread arbitrarily small open sets in all directions.
\end{remark}

The following is the main theorem of \cite{guiheneuf2023parabolic}:
\begin{theorem}
\label{thm:elliptic}
    Let $f \in \mathrm{Homeo}(\torus)$. Then $f$ acts elliptically on $\finegraph{\torus}$ if and only if it has bounded deviation in some rational direction $v \in \m Q^2 \setminus \{0\}$.
\end{theorem} 

\begin{proof}[Proof of Theorem \ref{thm:parabolic-nonproper}]
    By Theorem \ref{thm:weakspreading}, weak spreading maps are generic in $\obartor$. Therefore by Proposition \ref{prop:gen-rig}, a generic element of $\obartor$ is weak-spreading and rigid. \par
    \begin{sloppypar}
        Let $f$ be a rigid weak spreading homeomorphism of $\torus$. We show that $f$ acts parabolically and non-properly on the fine curve graph $\finegraph \torus$. Let $\alpha,\beta:\m R/\m Z \to \torus \simeq \rtwo / \ztwo$ be two curves on the torus parameterized as ${\alpha(t)=[(0,t)]}$ and $\beta(t)=[(1/2,t)]$. These closed curves are disjoint, simple, and non-contractible. Therefore by rigidity of $f$, there exists an increasing sequence of positive integers $(n_i)_{i \in \mathbb{N}}$ such that $f^{n_i}(\alpha)$ is disjoint from $\beta$. Hence, $\alpha$ and $f^{n_i}(\alpha)$ are at distance at most $2$ in $\finegraph{\torus}$. So $(f^{n_i}(\alpha))_i$ is bounded in the fine curve graph, and the action of $f$ on $\finegraph{\torus}$ is non-proper.
    \end{sloppypar}
    As a consequence, the asymptotic translation length is zero, so this action is not hyperbolic. On the other hand, weak spreading maps have unbounded deviation in all directions. Therefore by Theorem \ref{thm:elliptic}, $f$ cannot act elliptically so it must act parabolically.
\end{proof}
\begin{remark}
    Let $f \in \mrm{Diff}\toinfty(\torus)$ be rigid and weak spreading, and $\Tilde{f}$ a lift to $\rtwo$. Then $\liminf_{n \to \infty} \diam(\Tilde{f}^n(D)) < +\infty$, and $\limsup_{n \to \infty} \diam(\Tilde{f}^n(D)) = +\infty$. Therefore a generic map in $\obartor$ has this property. This answers a question posed in \cite{guiheneuf2023parabolic} about the existence of such torus homeomorphisms.
\end{remark}
\subsection{Generic rotation sets in the Anosov-Katok class}
A torus homeomorphism is called a \emph{pseudo rotation} if its rotation set consists of a single vector (which we refer to as the map's rotation vector). Any rigid homeomorphism of the torus homotopic to the identity is a pseudo rotation. This is true since for any such homeomorphism, there exists a subsequence of the iterates of the fundamental domain by a lift such that each of those iterates has bounded diameter. So the following corollary is a direct result of Proposition \ref{prop:gen-rig}.
\begin{corollary}
\label{cor:gen-pr}
    Pseudo rotations are generic in $\obartor$ and $\obartormes$.
\end{corollary}
We can say more about these generic rotation sets. A vector in $\rtwo$ is called \emph{totally Liouvillean} if both its coordinates and its slope are Liouville numbers. A real number $\alpha$ is said to be Liouville if for any $k \in \nstar$, there exists a pair of integers $p,q \in \m Z$ with $q > 1$ such that
\begin{equation*}
    0 < \left|\alpha - \frac{p}{q}\right| < \frac{1}{q^k}.
\end{equation*}
\begin{proposition}
    The rotation vector of a generic diffeomorphism in $\obartor$ (similarly in $\obartormes$) is totally Liouvillean.
\end{proposition}
\begin{proof}
    Let $Q$ be as in equation \eqref{eq:Q}, and for each $k \in \nstar$, let $P_k$ be as in equation \eqref{eq:Pk}, in the proof of Proposition \ref{prop:gen-rig}. Denote $\cap_k P_k$ by $P$. As we saw in Proposition \ref{prop:gen-rig}, $P$ is residual in $\obartor$ and all its elements are rigid, and hence pseudo rotations. The set $Q$ is dense in $\obartor$, and we have $Q \subseteq P$. Note that all rotation vectors of diffeomorphisms in $Q$ are rational. Let $\Tilde{Q}$ denote the set of diffeomorphism in $Q$ whose rotation vectors are not vertical, do not have integer coordinates, or integer slopes. The set $\Tilde{Q}$ is also dense in $\obartor$. \par
    For a diffeomorphism $f \in P$, denote the first coordinate of its rotation vector by $\rho_1(f)$. For a rational number $r$, denote the positive denominator of its simplest form by $\den(r)$. Then, for each $k \in \nstar$, we define
    \begin{equation*}
        U_k = \bigcup_{g \in \Tilde{Q}} \left\{f \in P \mid 0 < |\rho_1(f)-\rho_1(g)|<\frac{1}{\den(\rho_1(g))^k}\right\}.
    \end{equation*}
    A theorem of Misiurewicz and Ziemian \cite{misiurewicz1989rotation} states that the map sending a homeomorphism to its rotation set is upper semi-continuous, so this map is continuous on pseudo rotations. Hence the rotation vector varies continuously on $P$. Thus the sets $U_k$ are open subsets of $P$. By density of $P$ in $\obartor$, we have that for each $k$, any $g \in \Tilde{Q}$ is accumulated by elements of $U_k$. Therefore $U_k$ is dense in $\obartor$ for each $k$. Since $P$ is residual, we conclude that $U = \cap_{k} U_k$ is also residual in $\obartor$. By the definition of Liouville numbers, for any $f \in U$, the first coordinate of the rotation vector $\rho_1(f)$ is Liouvillean, so this property is generic in $\obartor$. We can see in the same way that generically, the second coordinate of the rotation set is also Liouvillean in $\obartor$. For the slope, we can replace $\rho_1(f)$ by the slope of $\rho(f)$ in the definition of $U_k$ and conclude in the same way. Thus we have shown that generically in $\obartor$, the rotation vector is totally Liouvillean. 
\end{proof}
\begin{remark}
    In fact, stronger Liouvillean properties are also generic in this class. Let $\varphi: \nstar \times (\nstar \setminus \set 1) \to \m R_+$ be a function. Then we say a real number $\alpha$ is \emph{$\varphi$-Liouville} if for any $k \in \nstar$, there exists a pair of integers $p,q \in \m Z$ with $q > 1$ such that
    \begin{equation*}
        0 < \left|\alpha - \frac{p}{q}\right| < \frac{1}{\varphi(k,q)}.
    \end{equation*}
    This definition is interesting when the function tends to infinity in both its coordinates. By taking the function $\varphi: (k,q) \mapsto q^k$, we recover the classic Liouville numbers. But in this way we may also describe stronger approximation with rationals by taking for example $\varphi: (k,q) \mapsto e^{kq}$.\par
    By following the above proof and adjusting the definition of $U_k$ we can show that for any $\varphi: \nstar \times (\nstar \setminus \set 1) \to \m R_+$, the rotation vector of a generic diffeomorphism in $\obartor$ (similarly in $\obartormes$) is totally $\varphi$-Liouvillean.
\end{remark}
\section{Generic homothety types of generalized rotation sets in the Anosov-Katok class}
\label{sec:generic-shapes}
In this section we discuss the homothety types of generalized rotation sets of generic elements in the Anosov-Katok class. We only consider homotheties with positive ratio. Let us clarify what we mean by homothety type.
\begin{definition}
\label{def:homothety-type}
    The quotient of the set of non-empty compact subsets of $\rtwo$ by the action of the group generated by homotheties and translations is called the set of \emph{homothety types} of $\rtwo$, and is denoted by $\shapes$. The class of a non-empty compact $K \subseteq \rtwo$ in this quotient is called its \emph{homothety type} and is denoted by $[K]$.
\end{definition}
Our goal is to prove the following theorem:
\begin{theoremb}
    A generic diffeomorphism in $\obartor$ (similarly in $\obartormes$) admits generalized rotation sets of all point-symmetric compact convex homothety types in $\rtwo$.
\end{theoremb}
Often in proofs using the Anosov-Katok method, the desired property for a diffeomorphism is thought of as an accumulation of countably many approximations of that property. These approximations are usually more manageable, and we can construct conjugates of rotations with these approximate properties. To this end, we introduce the notion of large approximates (see Definiton \ref{def:large-app}) which measure how close a domain is to a certain homothety type. Therefore by understanding how to construct and manipulate large approximates we can approach the problem of forms of generalized rotation sets.\par
In Subsection \ref{sec:fathi-herman}, we explain this in detail, and bring the strategy for proving Theorem \ref{thm:generic-all-symmetric} using the Fathi-Herman method. The core of this proof consists of Lemmas \ref{lem:one} and \ref{lem:two}. The former explores the effects of conjugation by a diffeomorphism on large approximates for the iterates of the fundamental domain by that diffeomorphism. For the most part, this relies on Lemmas \ref{lem:perturbation} and \ref{lem:linear-mapping} which show the effects of perturbations and application of linear maps to large approximates. The proof of Lemma \ref{lem:one} is brought in Subsection \ref{sec:proof-one}.\par
\begin{figure}[ht]
    \centering
    \includegraphics[width=0.8\textwidth]{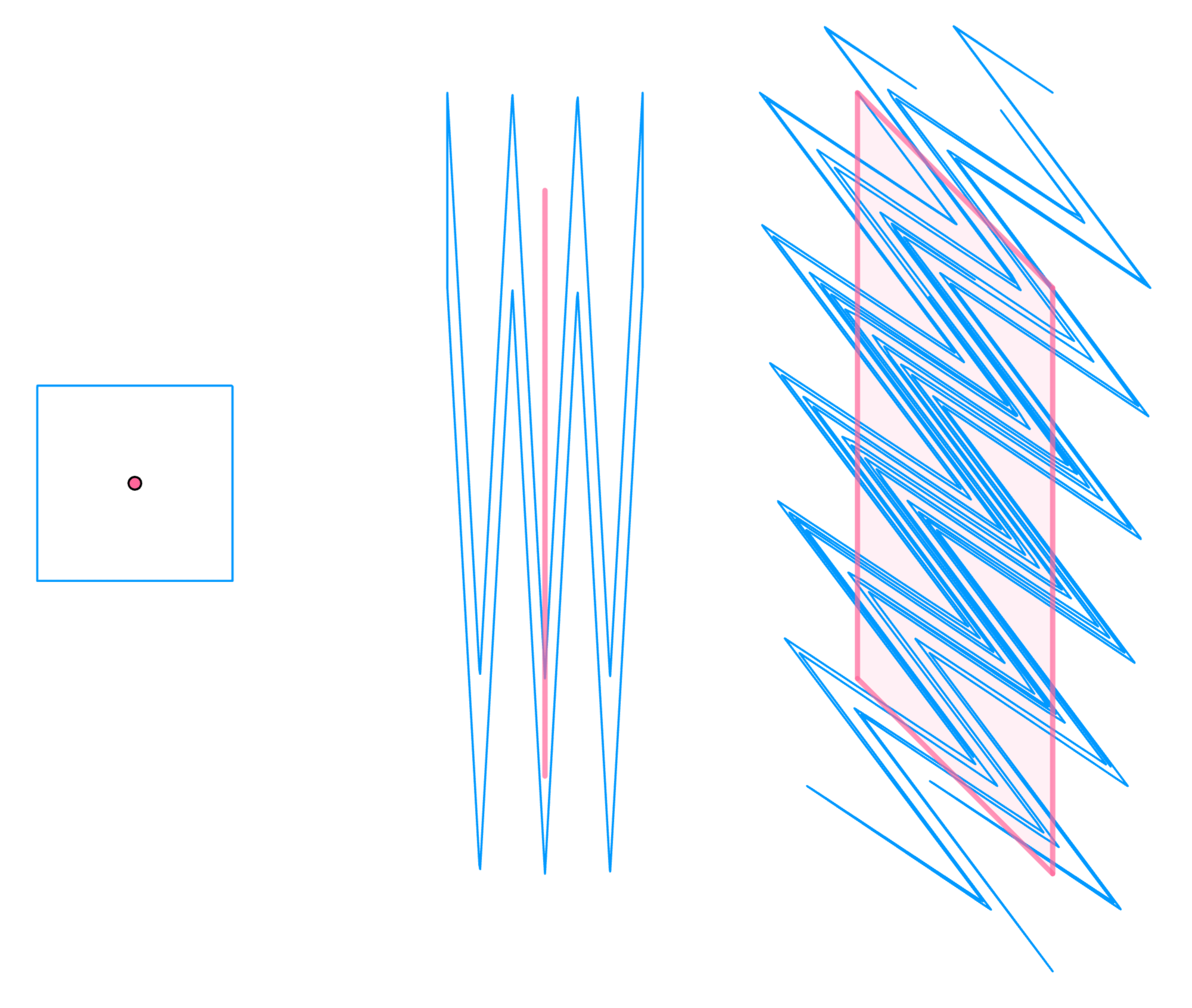}
    \caption{The construction of a homeomorphism spreading a fundamental domain to an approximation of a zonogon}
    \label{fig:construction}
\end{figure}
Any point-symmetric compact convex subset of the plane can be seen as a limit of point-symmetric convex polygons, also known as zonogons. Since there are countably many rational zonogons, we can first restrict ourselves to such forms and deduce the general result from this one. In Lemma \ref{lem:two}, we construct a conjugate of a rigid translation that commutes with a fixed rational translation, and such that an iterate of the fundamental domain by one of its lifts is a sufficiently accurate large approximate of a zonogon. The last section of this article is dedicated to this construction. The main idea is to work with the product of homeomorphisms which spread domains in the directions of edges of the zonogon. Figure \ref{fig:construction} shows how this works.

\subsection{A few definitions and notations}
\label{sec:def-not}
Denote the Hausdorff distance by $d_H$. The set of non-empty compact subsets of $\rtwo$ endowed with this distance is a complete metric space.
\subsubsection{Large approximates}
Let us give a sense to a compact approximating a homothety type.
\begin{definition}
\label{def:large-app}
    Let $r>0$, $\Gamma \in \shapes$, and $K$ a non-empty compact subset of $\rtwo$. Then $K$ is said to be an \emph{$r$-large approximate} of $\Gamma$ if $\diam(K) > r$, and if there exists $\hat{\Gamma} \in \Gamma$ of diameter $1$ such that
    \begin{equation*}
        d_H(\frac{K}{\diam(K)},\hat{\Gamma}) < \frac{1}{r}.
    \end{equation*}
    The set of all $r$-large approximates of $\Gamma$ is denoted by $\la{r}{\Gamma}$.
\end{definition}
The following lemma shows that in a sense large approximates are stable under small perturbations, up to losing some approximation accuracy.
\begin{lemma}
\label{lem:perturbation}
    Let $r > 0$, and $d_0 > 0$. Then there exists $s_0$ such that for all $s \geq s_0 $, for any $\Gamma \in \shapes$ and $K,K' \subseteq \rtwo$ non-empty compact subsets such that $d_H(K,K') < d_0$, if $\isla{K}{s}{\Gamma}$, then $\isla{K'}{r}{\Gamma}$.
\end{lemma}
\begin{proof}
    Fix any $s \geq 2d_0 + r + 3d_0 r$. Let $\Gamma \in \shapes$, and $K,K'$ two compact subsets of $\rtwo$ with Hausdorff distance less than $d_0$. Denote the diameters of $K$ and $K'$ by $\delta_K$ and $\delta_{K'}$ respectively. So $|\delta_K-\delta_{K'}| < 2d_0$. Assuming $\isla{K}{s}{\Gamma}$, we have $\delta_K>s$, and there exists $\hat{\Gamma} \in \Gamma$ of diameter $1$ such that $d_H(K/\delta_K,\hat{\Gamma})<{1}/{s}$, or equivalently,
    \begin{equation}
    \label{eq:hyp-K}
        d_H(K,\delta_K\hat{\Gamma})<{\delta_K}/{s}.
    \end{equation}
    Choose an arbitrary point $x_0$ in $\hat{\Gamma}$. The set $\hat{\Gamma} - x_0$ is included in the closed ball of radius $1$ around the origin. Therefore,
    \begin{equation*}
        d_H(\delta_K(\hat{\Gamma}-x_0),\delta_{K'}(\hat{\Gamma}-x_0)) \leq |\delta_K-\delta_{K'}| < 2d_0.
    \end{equation*}
    Denote the vector $- \delta_{K'}x_0 + \delta_Kx_0$ by $\nu$. Then by translating the above sets with vector $\delta_Kx_0$, we get
    \begin{equation}
    \label{eq:nu}
        d_H(\delta_K\hat{\Gamma},\delta_{K'}\hat{\Gamma} + \nu) < 2d_0.
    \end{equation}
    By the triangle inequality, we have
    \begin{equation*}
        d_H(K',\delta_{K'}\hat{\Gamma}+\nu) \leq d_H(K',K) + d_H(K,\delta_K\hat{\Gamma}) + d_H (\delta_K\hat{\Gamma},\delta_{K'}\hat{\Gamma}+\nu).
    \end{equation*}
    Therefore by equations \eqref{eq:hyp-K} and \eqref{eq:nu},
    \begin{equation}
    \label{eq:first}
        d_H(K',\delta_{K'}\hat{\Gamma}+\nu)< \frac{\delta_K}{s}+3d_0.
    \end{equation}
    
    Given that $|\delta_K-\delta_{K'}| < 2d_0$,
    \begin{equation}
    \label{eq:second}
        \delta_{K'} > \delta_K - 2d_0 > s - 2d_0 > r.
    \end{equation}
    On the other hand, by multiplying the inequalities $s < \delta_K$ and ${2d_0+3d_0 r \leq s-r}$, we get $s(2d_0+3d_0 r) < \delta_K(s-r)$. By adding ${\delta_Kr-2sd_0}$ to the sides of this inequality we have:
    \begin{equation*}
        r\delta_K+3sd_0 r < s\delta_K-2sd_0.
    \end{equation*}
    Dividing this inequality by $rs$ and using equation \eqref{eq:second} gives:
    \begin{equation*}
        \frac{\delta_K}{s} + 3d_0 < \frac{\delta_K-2d_0}{r} < \frac{\delta_{K'}}{r}.
    \end{equation*}
    Then by equation \eqref{eq:first},
    \begin{equation*}
    d_H(K',\delta_{K'}\hat{\Gamma}+\nu) < \frac{\delta_{K'}}{r}
        \implies d_H(K'/\delta_{K'},\hat{\Gamma}+(\nu/\delta_{K'})) < \frac{1}{r}.
    \end{equation*}
    The set $\hat{\Gamma}+(\nu/\delta_{K'})$ has diameter $1$ and homothety type $\Gamma$, so $\isla{K'}{r}{\Gamma}$.
\end{proof}
Remember that the operator norm of $A$, a linear mapping of $\rtwo$, can be calculated as follows: $$\opnorm{A} = \sup_{v \in \rtwo \setminus \set{0}} \frac{|A(v)|}{|v|}.$$
\begin{lemma}
\label{lem:linear-mapping}
    Let $r>0$, and $A \in \mrm{GL}(2,\m R)$. Then there exists $s_0$ such that for all $s \geq s_0 $, for any $\Gamma \in \shapes$ and any non-empty compact $K \subseteq \rtwo$ in $\la{s}{\Gamma}$, we have $\isla{A(K)}{r}{A(\Gamma)}$.
\end{lemma}
\begin{proof}
    Fix any $s \geq r \opnorm{A\inv} \cdot \max \set{3 \opnorm{A},1}$. Let $\Gamma \in \shapes$, and $K$ a compact subset of $\rtwo$. Denote the diameters of $K$ and $A(K)$ by $\delta_K$ and $\delta_{A(K)}$ respectively. If $\isla{K}{s}{\Gamma}$, then $\delta_K>s$, and there exists $\hat{\Gamma} \in \Gamma$ of diameter $1$ such that $d_H(K,\delta_K\hat{\Gamma})<{\delta_K}/{s}$. Denote the diameter of $A(\hat{\Gamma})$ by $\delta_{A(\hat{\Gamma})}$. \par
    By compactness of $K$, there exist $x,y \in K$ such that $|x-y|=\delta_K$. So,
    \begin{equation}
    \label{eq:fifth}
        \opnorm{A\inv}\inv\cdot\delta_K = \opnorm{A\inv}\inv\cdot |x-y| \leq |A(x-y)| = |A(x)-A(y)| \leq \delta_{A(K)}.
    \end{equation}
    Hence, $\delta_{A(K)} \geq \opnorm{A\inv}\inv\cdot\delta_K > \opnorm{A\inv}\inv\cdot s \geq r$. \par
    As $d_H(K,\delta_K\hat{\Gamma})<{\delta_K}/{s}$, for any $x \in K$, there exists $y \in \delta_K\hat{\Gamma}$ such that $|x-y|<{\delta_K}/{s}$, which implies $|A(x)-A(y)|<\opnorm{A}{\delta_K}/{s}$. Hence $A(K)$ is in the $(\opnorm{A}{\delta_K}/{s})$-neighborhood of $A(\delta_K\hat{\Gamma})=\delta_KA(\hat{\Gamma})$. Similarly, the reciprocal is also true. Therefore,
    \begin{equation}
    \label{eq:third}
        d_H(A(K),\delta_KA(\hat{\Gamma})) < \opnorm{A}\frac{\delta_K}{s}.
    \end{equation}
    This implies in particular that $|\delta_K\delta_{A(\hat{\Gamma})}-\delta_{A(K)}|<2\opnorm{A}{\delta_K}/{s}$. Therefore for an appropriate vector $\nu$:
    \begin{multline}
    \label{eq:fourth}
    d_H(\delta_K{A(\hat{\Gamma})}, \delta_{A(K)} \left(\tfrac{A(\hat{\Gamma})}{\delta_{A(\hat{\Gamma})}}\right)+\nu) = \\
    d_H(\delta_K\delta_{A(\hat{\Gamma})}\left(\tfrac{A(\hat{\Gamma})}{\delta_{A(\hat{\Gamma})}}\right), \delta_{A(K)} \left(\tfrac{A(\hat{\Gamma})}{\delta_{A(\hat{\Gamma})}}\right)+\nu)<2\opnorm{A}\frac{\delta_K}{s}.
    \end{multline}
    So, by the triangle inequality and equations \eqref{eq:fifth}, \eqref{eq:third} and \eqref{eq:fourth},
    \begin{equation*}
        d_H(A(K),\delta_{A(K)} \left(\tfrac{A(\hat{\Gamma})}{\delta_{A(\hat{\Gamma})}}\right)+\nu)<3\opnorm{A}\frac{\delta_K}{s} \leq 3\opnorm{A}\frac{\delta_K}{3r \opnorm{A}\opnorm{A\inv}} \leq \frac{\delta_{A(K)}}{r}.
    \end{equation*}
    Hence $\isla{A(K)}{r}{A(\Gamma)}$.
\end{proof}

\subsubsection{Zonogons}
\begin{definition}
    A point-symmetric convex polygon is called a \emph{zonogon}. A zonogon is called \emph{rational} if its vertices are rational points of the plane.
\end{definition}
\begin{figure}[ht]
    \centering
    \begin{subfigure}{0.4\textwidth}
        \centering
        \includegraphics[width=\textwidth]{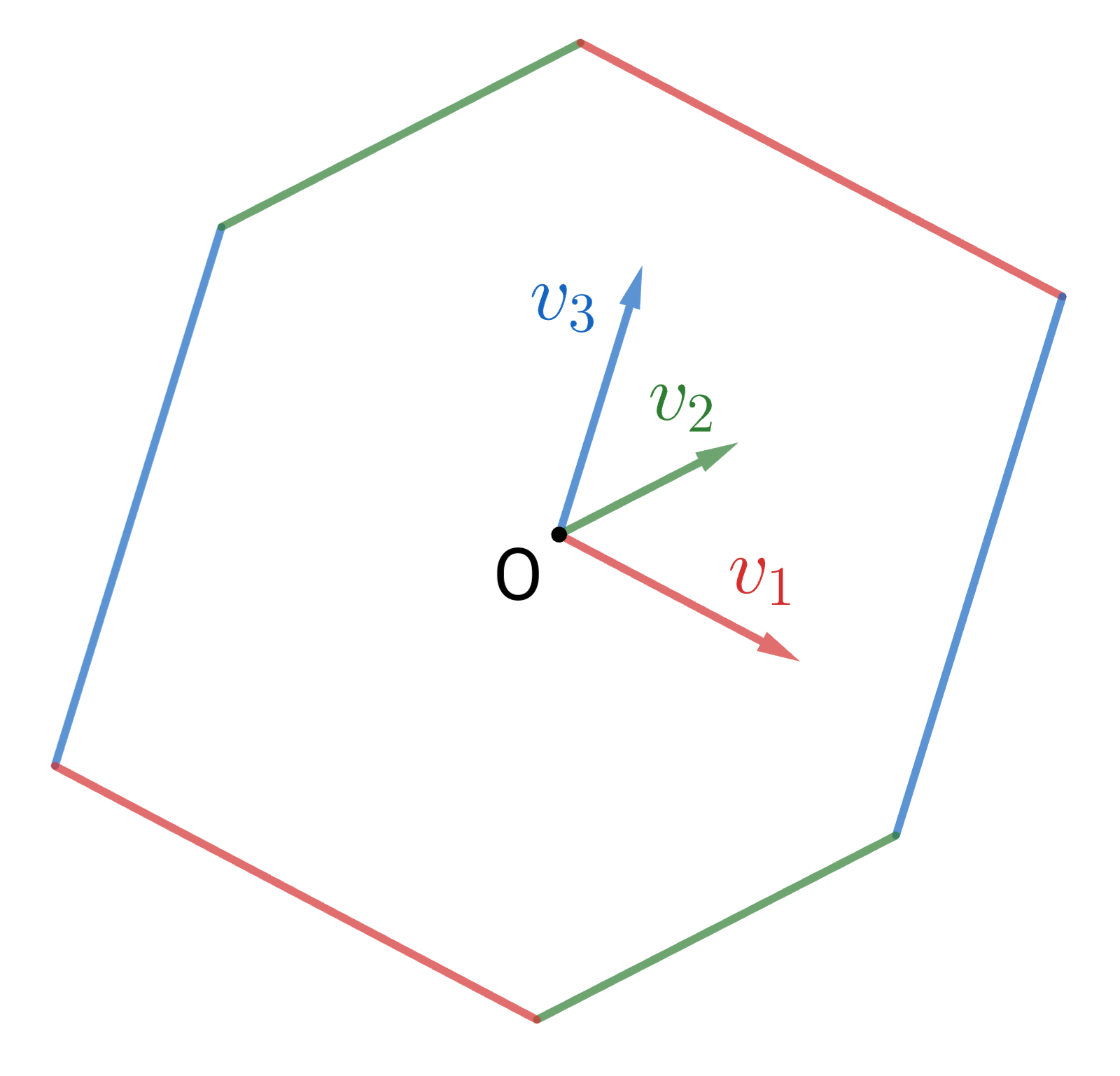}
        \caption{A picture of $\zon(v_1,v_2,v_3)$}
        \label{fig:zonogon}
    \end{subfigure}
    \begin{subfigure}{0.4\textwidth}
        \centering
        \includegraphics[width=\textwidth]{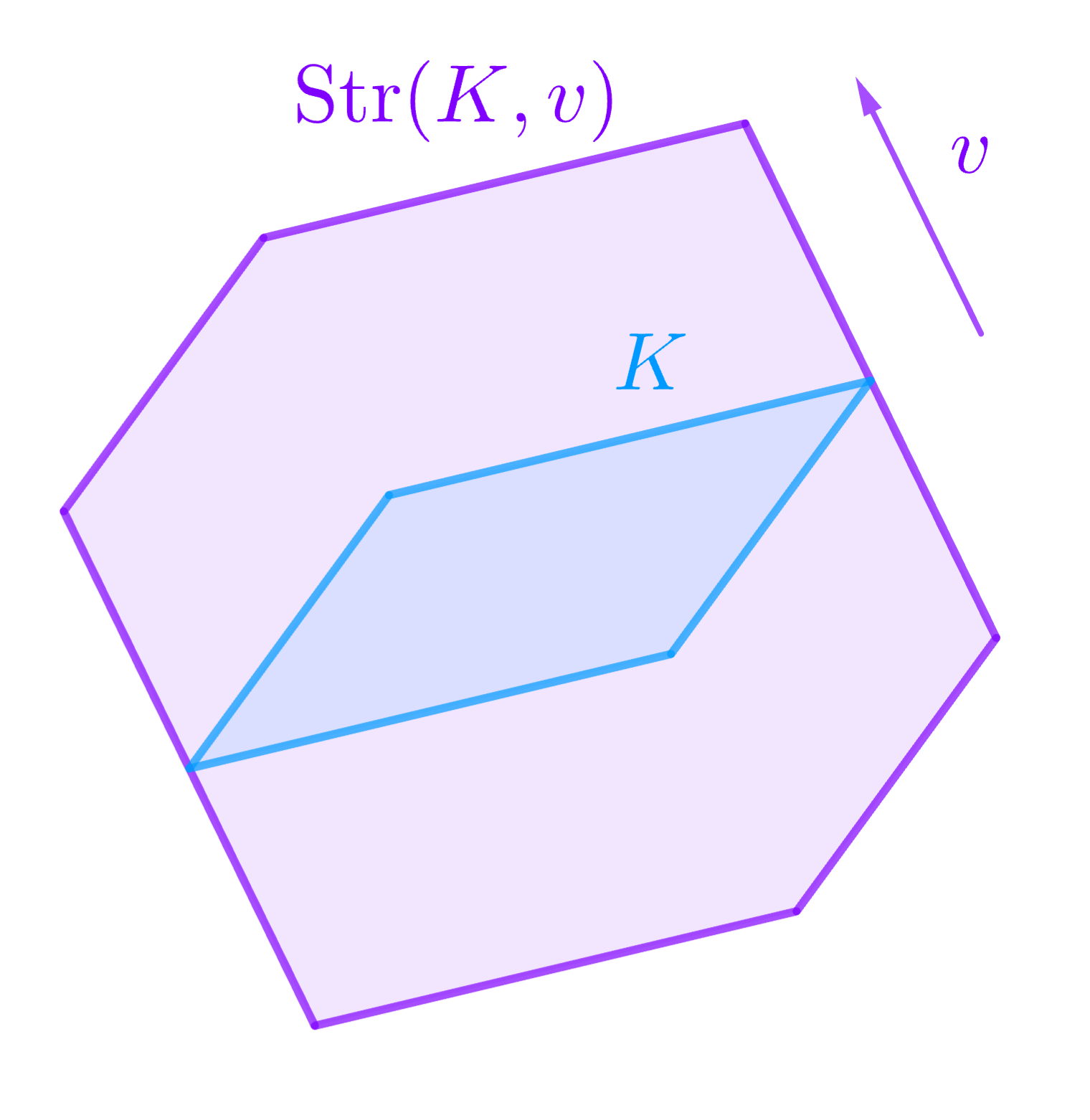}
        \caption{A set $K$ and its stretch in direction $v$}
        \label{fig:stretch}
    \end{subfigure}
        \caption{}
\end{figure}
For the purpose of this paper, we accept non-trivial segments as zonogons. A zonogon which is not a segment has evenly many sides, and the opposite sides are parallel and have equal length.\par
For any $v,w \in \rtwo$, the line segment joining $v$ and $w$ is denoted by $[v,w]$. Let $v_1,\dots,v_l$ be non collinear vectors in $\rtwo$. The Minkowski sum $$[-v_1,v_1]+\dots+[-v_l,v_l]$$ is a zonogon. Denote this zonogon as $\zon(v_1,\dots,v_l) \subseteq \rtwo$.
\begin{definition}
    Let $X \subseteq \rtwo$, and $v \in \rtwo \setminus \set{0}$. The Minkowski sum $X + [-v,v]$ is denoted by $\stre(X,v)$.
\end{definition}

\subsubsection{Construction of diffeomorphisms commuting with $R_{(1/q,0)}$}
For $q \in \nstar$, let $C_q$ denote the matrix $\begin{pmatrix}
    {1}/{q} & 0 \\
    0 & 1
\end{pmatrix}$. This linear map contracts the horizontal direction by a factor of $q$. The following lemma shows how we may use conjugation by $C_q$ to obtain a torus diffeomorphism commuting with the horizontal translation $R_{(1/q,0)}$ starting from a diffeomorphism homotopic to identity.
\begin{lemma}
\label{lem:conj-by-C}
    Let $h$ be a torus diffeomorphism that is homotopic to identity, and $q \in \nstar$. Then for $\Tilde{h}$ a lift of $h$, the plane diffeomorphism $\conj{C_q}{\Tilde{h}}$ quotients down to a torus diffeomorphism homotopic to identity, denoted with abuse of notation by $\conj{C_q}{h}$. Additionally, $\conj{C_q}{h}$ commutes with the translation map $R_{(1/q,0)}$.
\end{lemma}
\begin{proof}
\begin{figure}[ht]
    \centering
    \includegraphics[width=0.7\textwidth]{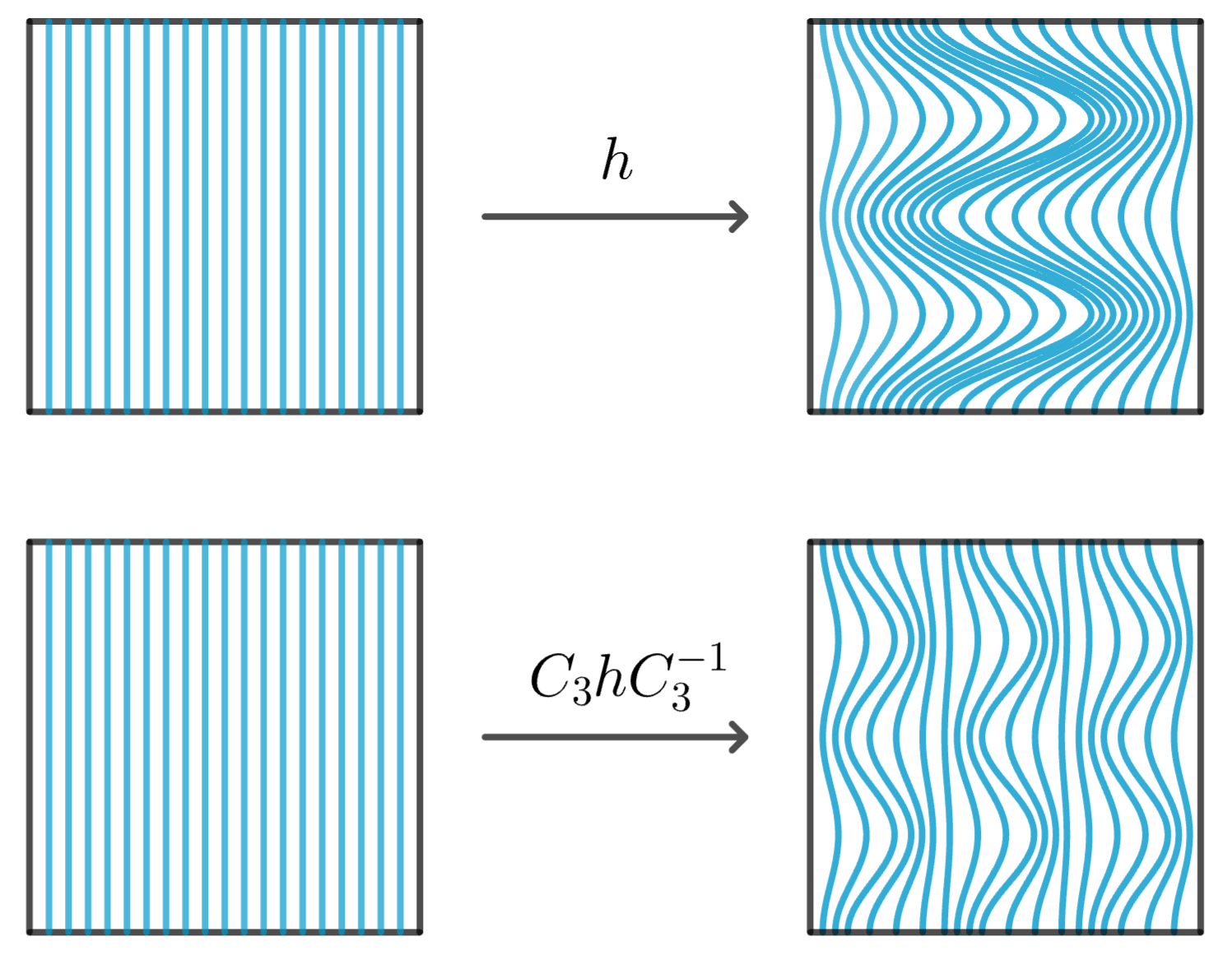}
    \caption{}
    \label{fig:commute}
\end{figure}
    Define $H \coloneqq \conj{C_q}{\Tilde{h}}$. The map $h$ is homotopic to identity, hence $\Tilde{h}$ commutes with integer translations. Therefore, for any $(x_1,x_2) \in \rtwo$, and $(\zeta_1,\zeta_2) \in \ztwo$:
    \begin{align*}
        H(x_1+\frac{\zeta_1}{q},x_2+\zeta_2)
        &= \conj{C_q}{\Tilde{h}}(x_1+\frac{\zeta_1}{q},x_2+\zeta_2) \\
        &= C_q \Tilde{h} (C_q\inv (x_1,x_2) + (\zeta_1,\zeta_2)) \\
        &= C_q(\Tilde{h} C_q\inv (x_1,x_2) + (\zeta_1,\zeta_2)) \\
        &= \conj{C_q}{\Tilde{h}}(x_1,x_2) + (\frac{\zeta_1}{q},\zeta_2) \\
        &= H(x_1,x_2) + (\frac{\zeta_1}{q},\zeta_2).
    \end{align*}
    So $H$ is a homeomorphism of $\rtwo$ which commutes with integer translations and $\Tilde{R}_{(1/q,0)}$. Hence it quotients down to a torus homeomorphism homotopic to identity and commuting with $R_{(1/q,0)}$.
\end{proof}

\subsection{The Fathi-Herman method and the proof}
\label{sec:fathi-herman}
\begin{sloppypar}
    Fix $D = [0,1]^2$, and $\Tilde{x}_0 \in D$. For each $k \in \nstar$ and $\Gamma \in \shapes$, we define ${P_k^{\Gamma} \subseteq \obartor}$ as follows:
\end{sloppypar}
\begin{equation*}
    P_k^{\Gamma} \coloneqq \set{f \in \obartor | \exists n \in \nstar : \isla{\Tilde{f}^n(D)}{k}{\Gamma}},
\end{equation*}
where $\Tilde{f}$ denotes a lift of $f$. Note that the condition $\isla{\Tilde{f}^n(D)}{k}{\Gamma}$ remains unchanged no matter which lift is chosen. Denote $P^{\Gamma} = \bigcap_{k \geq 1} P_k^{\Gamma}$. \par
\begin{lemma}
    Let $f$ be a diffeomorphism in $P^{\Gamma}$. Then $\Gamma$ is the homothety type of a generalized rotation set of $f$.
\end{lemma}
\begin{proof}
    Let $\Tilde{f}$ be a lift of $f$ to $\rtwo$. As $f$ is in $P^{\Gamma}$, there exists a sequence of natural numbers $(n_k)_{k \in \m N}$ such that for each $k \in \nstar$, $\Tilde{f}^{n_k}(D)$ is a $k$-large approximate of $\Gamma$. So $\diam(\Tilde{f}^{n_k}(D)) > k$, and there exists $\hat{\Gamma}_k \in \Gamma$ of diameter $1$ such that
    \begin{equation*}
        d_H(\frac{\Tilde{f}^{n_k}(D) - \Tilde{f}^{n_k}(\Tilde{x}_0)}{\diam(\Tilde{f}^{n_k}(D))},\hat{\Gamma}_k) < \frac{1}{k}.
    \end{equation*}
    Given that the origin is of distance at most $1/k$ from $\hat{\Gamma}_k$, there exists $\hat{\Gamma} \in \Gamma$ and an increasing sequence of positive integers $(k_i)_{i \in \m N}$ such that $\hat{\Gamma}_{k_i} \to \hat{\Gamma}$ as $i \to \infty$ for the Hausdorff topology. Therefore, in the Hausdorff topology
    \begin{equation*}
        \frac{\Tilde{f}^{n_{k_i}}(D) - \Tilde{f}^{n_{k_i}}(\Tilde{x}_0)}{\diam(\Tilde{f}^{n_{k_i}}(D))} \xrightarrow[]{i \to \infty} \hat{\Gamma}.
    \end{equation*}
    So $\Gamma$ is the homothety type of a generalized rotation set of $f$.
\end{proof}
To prove the genericity of maps admitting a generalized rotation set of a fixed homothety type $\Gamma \in \shapes$ in $\obartor$, it suffices to show that for each $k \in \nstar$, the set $P_k^{\Gamma} \subseteq \obartor$ is open and dense. Then $P^{\Gamma}$ would be a residual set of maps with generalized rotation sets of homothety type $\Gamma$. By continuity of the mapping $f \mapsto f^n$ in the $C^0$ topology on $\obartor$, the sets $P_k^{\Gamma}$ are open. The following lemmas allow us to show density of $P_k^{\Gamma}$ for a rational zonogon $\Gamma$. \par
The action of a linear mapping $A \in \mathrm{GL}(2,\m Z)$ on $\rtwo$ by $x \mapsto Ax$ induces a homeomorphism on the torus. This homeomorphism is denoted by $f_A$. 
\begin{lemma}
\label{lem:one}
    For any $\Gamma \in \shapes$, $k \in \nstar$, and any torus diffeomorphism $g$ homotopic to $f_A$ for some $A \in \mathrm{GL}(2,\m Z)$, there exists $l \in \nstar$ such that $$\conj{g}{P_l^{A\inv(\Gamma)}} = \set{\conj{g}{f} | f \in P_l^{A\inv(\Gamma)}} \subseteq P_k^{\Gamma}.$$
\end{lemma}

\begin{lemma}
\label{lem:two}
    Let $\ell \in \nstar$, $p/q \in \m Q$, and $\Gamma \in \shapes$ a rational zonogon. There exists $h \in \mathrm{Diff}\toinfty(\torus)$ such that
        \begin{itemize}
            \item $\conj{h}{R_{({1}/{q},0)}} = R_{({1}/{q},0)}$, and
            \item There exists a sequence of vectors $({\theta}_i)_{i\in \m N}$ such that ${\theta}_i \to ({p}/{q},0)$, and $\conj{h}{R_{{\theta}_i}} \in P_{\ell}^{\Gamma}$.
        \end{itemize}
\end{lemma}

\begin{proof}[Proof of Theorem \ref{thm:generic-all-symmetric}]
    Fix $\Gamma \in \shapes$ a rational zonogon. We will show that $P^{\Gamma}$ is a residual subset of $\obartor$. As mentioned above, for each $k \in \nstar$ the set $P_k^{\Gamma}$ is open. It remains to show the density. Fix $k \in \nstar$. The set $$\{ \conj{g}{R_{(p/q,0)}} \mid g \in \mathrm{Diff}\toinfty(\torus),p/q \in \m Q \}$$ is dense in $\obartor$. To prove that $P_k^{\Gamma}$ is dense, it suffices to show that $\conj{g}{R_{(p/q,0)}} \in \cinfclosure{P_k^\Gamma}$ for any $g \in \mathrm{Diff}\toinfty(\torus)$ and $p/q \in \m Q$. \par
    Fix such $g$ and $p/q$. There exists $A \in \mathrm{GL}(2,\m Z)$ such that $f_A$ is homotopic to $g$. Take $l$ as provided by Lemma \ref{lem:one} for $\Gamma$, $k$, and $g$. By Lemma \ref{lem:two} applied to $l$, $p/q$, and $A\inv \Gamma$ which is also a rational zongon, there exists a sequence of diffeomorphisms $(\conj{h}{R_{{\theta}_i}})_i$ in $P_l^{A\inv(\Gamma)}$ converging to $R_{(p/q,0)}$. If we conjugate the diffeomorphisms of this sequence by $g$, using Lemma \ref{lem:one} we obtain a sequence of diffeomorphisms in $P_k^{\Gamma}$ converging to $\conj{g}{R_{(p/q,0)}}$. This concludes the proof of genericity of $P^{\Gamma}$. \par

    There are countably many rational zonogons. Therefore a generic element in $\obartor$ has a generalized rotation set of homothety type $\Gamma$, for any rational zonogon $\Gamma$. Any point-symmetric compact convex subset of $\rtwo$ is the limit of a sequence of rational zonogons in Hausdorff topology. As the set of generalized rotation sets of a map is closed in Hausdorff topology, we conclude that a generic element in $\obartor$ admits generalized rotation sets of all homothety type of point-symmetric compact convex subsets of $\rtwo$.
\end{proof}

\subsection{Proof of Lemma \ref{lem:one} and a similar result}
\label{sec:proof-one}
\begin{proof}[Proof of Lemma \ref{lem:one}]
\begin{figure}[ht]
    \centering
    \includegraphics[width=0.8\textwidth]{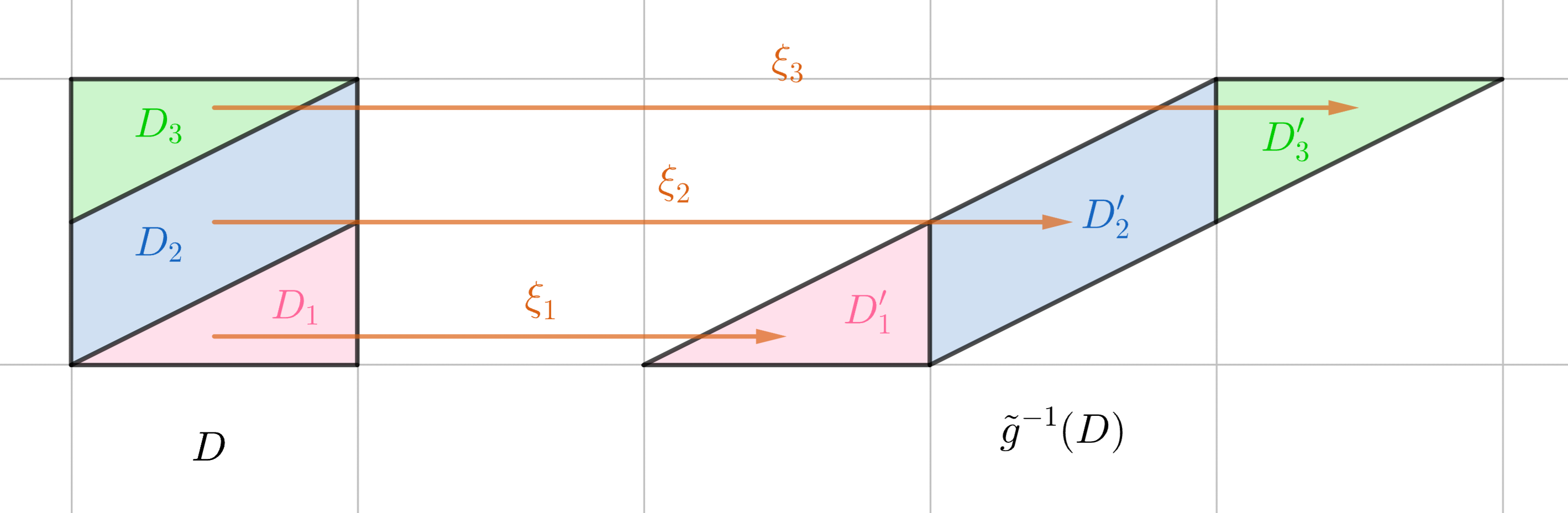}
    \caption{}
    \label{fig:puzzle}
\end{figure}
    Fix $\Tilde{g}$ a lift of $g$ to the plane. The plane homeomorphism $\Tilde{\phi} \coloneqq \Tilde{g} \circ A\inv$ commutes with integer translations, so it has bounded distance from $\mathrm{Id}_{\rtwo}$. Note that $\Tilde{g}\inv(D)$ is a fundamental domain of the action of $\ztwo$ on $\rtwo$ by translations. Let $\set{\zeta_i}_{i=1}^q = \set{\zeta \in \m Z^2 | \Tilde{g}\inv(D) \cap (\zeta+D) \neq \varnothing}$. For each $1 \leq i \leq q$, denote the set $\Tilde{g}\inv(D) \cap (\zeta_i+D)$ by $D'_i$, and denote the set $-\zeta_i+D'_i$ by $D_i$. As $D$ is a fundamental domain, $\Tilde{g}\inv(D) = \cup_{i=1}^{q} D'_i$, and as $\Tilde{g}\inv(D)$ is a fundamental domain, $D = \cup_{i=1}^{q} D_i$ (see Figure \ref{fig:puzzle}). Denote $Z = \max_i\set{|\zeta_i|} + 1$. \par
    We will obtain $l$ in three steps. First, denote by $l_1$ the lower bound ($s>0$) provided by Lemma $\ref{lem:perturbation}$ for $r = k$, and $d_0=d_{C^0}(\Tilde{\phi},\mrm{Id}_{\rtwo})$. Next, apply Lemma $\ref{lem:linear-mapping}$ to $r=l_1$ and $A$ to obtain the bound $l_2$. Finally, by applying Lemma $\ref{lem:perturbation}$ again for $r=l_2$ and $d_0=Z$ one obtains the bound $l$. Let $f \in P_l^{A\inv(\Gamma)}$ and let $\Tilde{f}$ be a lift. By definition, there exists $n \in \nstar$ such that $\Tilde{f}^n(D)$ is an $l$-large approximate of $A\inv(\Gamma)$. We will show that $\Tilde{g} \circ \Tilde{f}^n \circ \Tilde{g}\inv(D)$ is a $k$-large approximate of $\Gamma$. As a result, $\conj{g}{f}$ will fall in $P_k^{\Gamma}$. \par
    Remark that $P_l^{A\inv(\Gamma)}$ is a subset of $\obartor$ all of whose elements are homotopic to identity on the torus. Therefore $\Tilde{f}$ commutes with integer translations.
    \begin{multline*}
        \Tilde{f}^n(D) = \Tilde{f}^n\left(\bigcup_{i=1}^q D_i\right) = \bigcup_{i=1}^q \Tilde{f}^n(D_i) = \bigcup_{i=1}^q \Tilde{f}^n(-\zeta_i + D'_i) \\ = \bigcup_{i=1}^q (-\zeta_i + \Tilde{f}^n(D'_i)) \subseteq B_Z(\Tilde{f}^n(\Tilde{g}\inv(D))).
    \end{multline*}
    Similarly,
    \begin{multline*}
        \Tilde{f}^n(\Tilde{g}\inv(D)) = \Tilde{f}^n\left(\bigcup_{i=1}^q D'_i\right) = \bigcup_{i=1}^q \Tilde{f}^n(D'_i) = \bigcup_{i=1}^q \Tilde{f}^n(\zeta_i+D_i) \\ = \bigcup_{i=1}^q (\zeta_i+\Tilde{f}^n(D_i)) \subseteq B_Z(\Tilde{f}^n(D)).
    \end{multline*}
        So the Hausdorff distance between $\Tilde{f}^n(\Tilde{g}\inv(D))$ and $\Tilde{f}^n(D)$ is at most $Z$. Given that ${\isla{\Tilde{f}^n(D)}{l}{A\inv(\Gamma)}}$, this implies ${\isla{\Tilde{f}^n(\Tilde{g}\inv(D))}{l_2}{A\inv(\Gamma)}}$. Therefore ${\isla{A \circ \Tilde{f}^n \circ \Tilde{g}\inv(D)}{l_1}{\Gamma}}$. Finally, we can conclude that the domain $${\Tilde{\phi}\circ A \circ \Tilde{f}^n \circ \Tilde{g}\inv(D)=\Tilde{g} \circ \Tilde{f}^n \circ \Tilde{g}\inv(D)}$$ is a $k$-large approximate of $\Gamma$.
\end{proof}
The following lemma has a similar proof and will be used later.
\begin{lemma}
\label{lem:C}
    For any $\Gamma \in \shapes$, $q \in \nstar$, and $k \in \nstar$, there exists $l \in \nstar$ such that $\conj{C_q}{P_l^{C_q\inv(\Gamma)}}= \set{\conj{C_q}{f} | f \in P_l^{C_q\inv(\Gamma)}} \subseteq P_k^{\Gamma}$.
\end{lemma}

\begin{proof}
    Let's start by determining $l$. Denote by $l_1$ the bound provided by Lemma $\ref{lem:linear-mapping}$ for $r = k$, and $A=C_q$. Then, apply Lemma $\ref{lem:perturbation}$ to $r=l_1$ and $d_0=q$ to obtain the bound $l$. \par
    Let $f \in P_l^{C_q\inv(\Gamma)}$ and let $\Tilde{f}$ be any lift of $f$. Then there exists $n \in \nstar$ such that $\isla{\Tilde{f}^n(D)}{l}{C_q\inv(\Gamma)}$. We have
    \begin{equation*}
        C_q\inv(D)=C_q\inv([0,1]^2)=[0,q] \times [0,1] = \bigcup_{i=0}^{q-1} \big((i,0)+D \big).
    \end{equation*}
    The map $\Tilde{f}$ commutes with integer translations, so
    \begin{equation*}
        \Tilde{f}^n(C_q\inv(D)) = \Tilde{f}^n \left(\bigcup_{i=0}^{q-1} \big((i,0)+D \big) \right) = \bigcup_{i=0}^{q-1} \left((i,0)+\Tilde{f}^n(D)\right).
    \end{equation*}
    Therefore, the Hausdorff distance between $\Tilde{f}^n(C_q\inv(D))$ and $\Tilde{f}^n(D)$ is less than $q$. As we have $\isla{\Tilde{f}^n(D)}{l}{C_q\inv(\Gamma)}$, this implies $\isla{\Tilde{f}^n(C_q\inv(D))}{l_1}{C_q\inv(\Gamma)}$. Therefore $\isla{C_q \circ \Tilde{f}^n \circ C_q\inv(D)}{k}{\Gamma}$. This implies $\conj{C_q}{f} \in P_k^{\Gamma}$.
\end{proof}

\subsection{Proof of Lemma \ref{lem:two}}
\begin{sloppypar}
    Let $\xi \in \nstar$. The triangle wave function of period $1/\xi$ is denoted by ${\phi_{\xi}:\m R \to [-1,1]}$:
\end{sloppypar}
\begin{equation*}
    \phi_{\xi}(x)=(-1)^{z+1}(4\xi x - 2z - 1), \quad z \in \m Z,x \in [\frac{z}{2\xi},\frac{z+1}{2\xi}[
\end{equation*}
Let $\eta \in \m R^{*}$. The following formula defines a homeomorphism of $\rtwo$:
\begin{equation*}
    J_{\eta,\xi}(x,y)=(x,y+\eta \phi_{\xi}(x))
\end{equation*}
\begin{figure}[ht]
    \centering
    \includegraphics[width=0.6\textwidth]{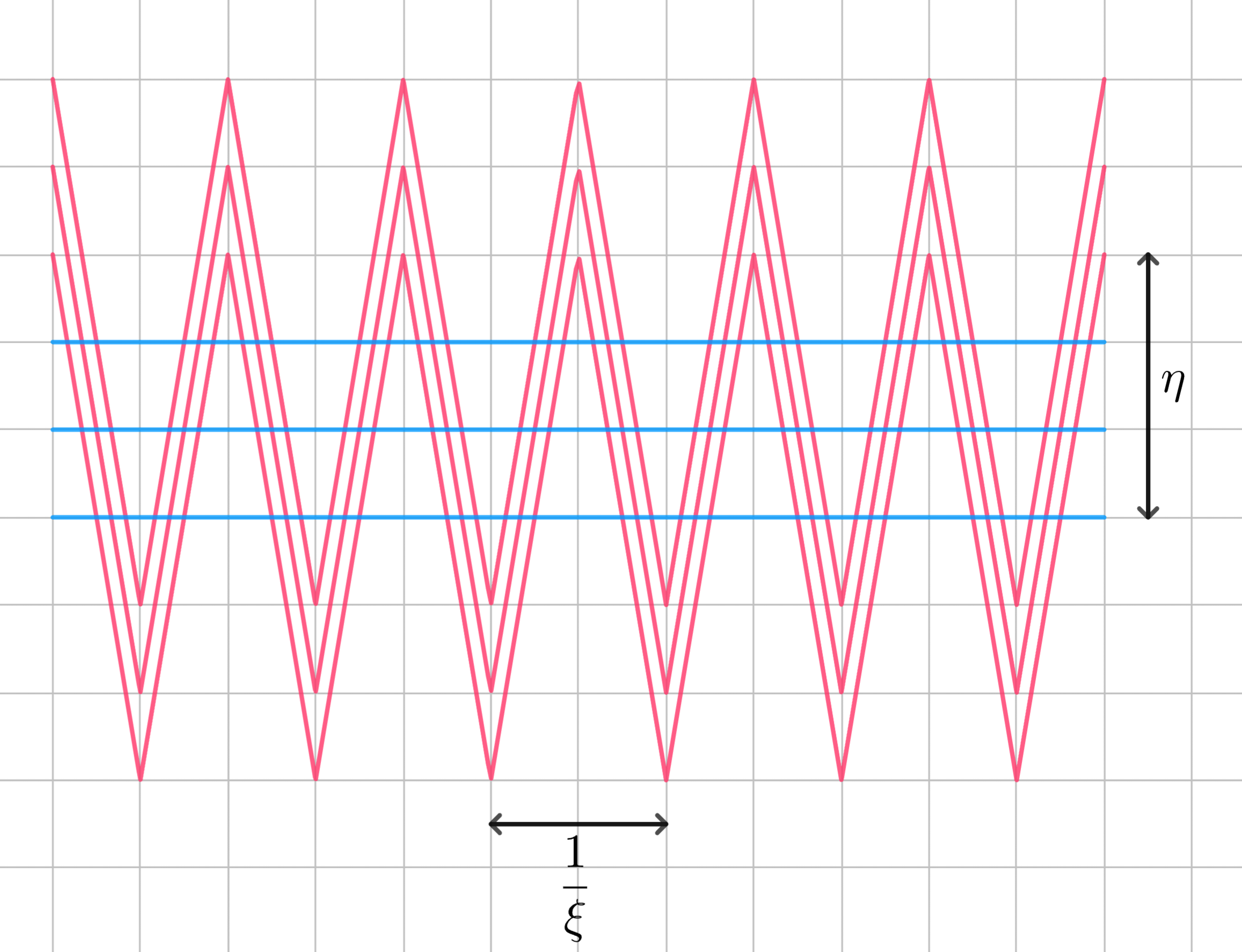}
    \caption{The homeomorphism $J_{\eta,\xi}$ sends the blue lines to the pink piecewise affine curves.}
    \label{fig:jagged}
\end{figure}
This map preserves vertical lines. On each line $x=x_0$, the dynamics is a translation by a number in the interval $[-|\eta|,|\eta|]$, and this number is $1/{\xi}$-periodic in $x_0$. In particular, $J_{\eta,\xi}$ commutes with integer translations and quotients down to a torus homeomorphism homotopic to identity. Note that $J_{\eta_1,\xi}\circ J_{\eta_2,\xi}=J_{\eta_1+\eta_2,\xi}$ and $J_{\eta,\xi}\inv=J_{-\eta,\xi}$.\par
The length of an affine segment of the plane $I$ is denoted by $|I|$.
\begin{lemma}
\label{lem:spread-shape}
Let $X$ be a non-empty subset of $\rtwo$. Fix $\eta \in \m R^{*}$, and let $\varepsilon,\varepsilon',\ell,m,$ and $M$ be positive real numbers. Let $\gamma$ be a union of segments $\cup_{i \in S}I_i$ in $\rtwo$ with the following properties:
\begin{itemize}
    \item $\gamma$ is $\varepsilon$-dense in $X$,
    \item for all $i \in S$, the norm of the slope of $I_i$ is less than $M$, and
    \item for all $i \in S$, the length of $I_i$ is greater than $\ell$.
\end{itemize}
Then there exists $\xi_0=\xi_0(|\eta|,\varepsilon',\ell,m,M) \in \nstar$ such that for any $\xi \geq \xi_0$, there exists a subset $\gamma' \subseteq J_{\eta,\xi}(\gamma)$ with the following properties:
\begin{itemize}
    \item $\gamma'$ is a union of segments,
    \item $\gamma'$ is ($\varepsilon+\varepsilon'$)-dense in $\stre(X,(0,\eta))$,
    \item the norms of the slopes of the segments of $\gamma'$ are greater than $m$, and
    \item the lengths of the segments of $\gamma'$ are greater than $|\eta|$.
\end{itemize}
\end{lemma}

\begin{proof}
\begin{figure}[ht]
    \centering
    \begin{subfigure}{0.45\textwidth}
        \centering
        \includegraphics[width=\textwidth]{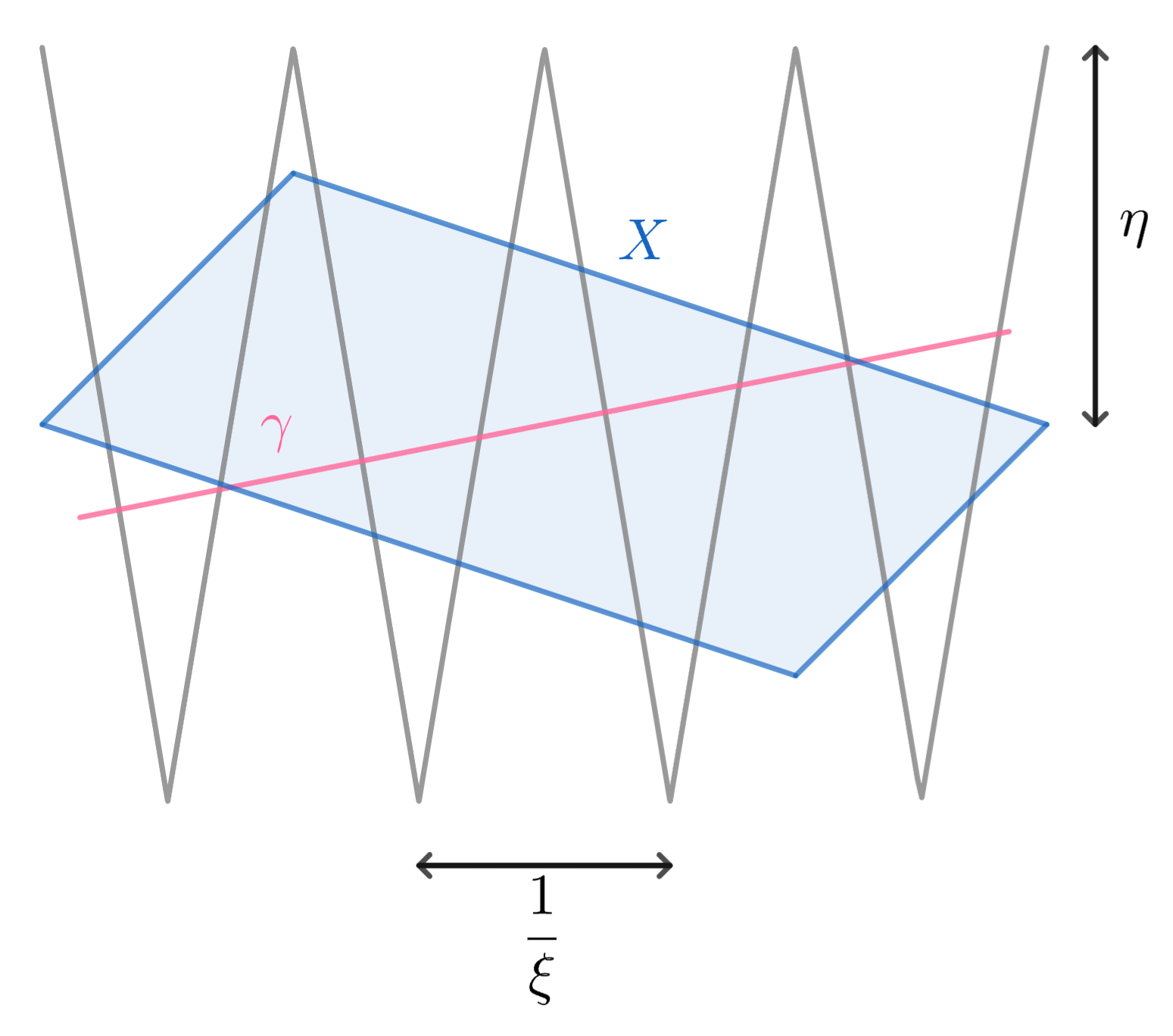}
        \caption{A segment $\gamma$ and a set $X$}
        \label{fig:before-stretch}
    \end{subfigure}
    \begin{subfigure}{0.45\textwidth}
        \centering
        \includegraphics[width=\textwidth]{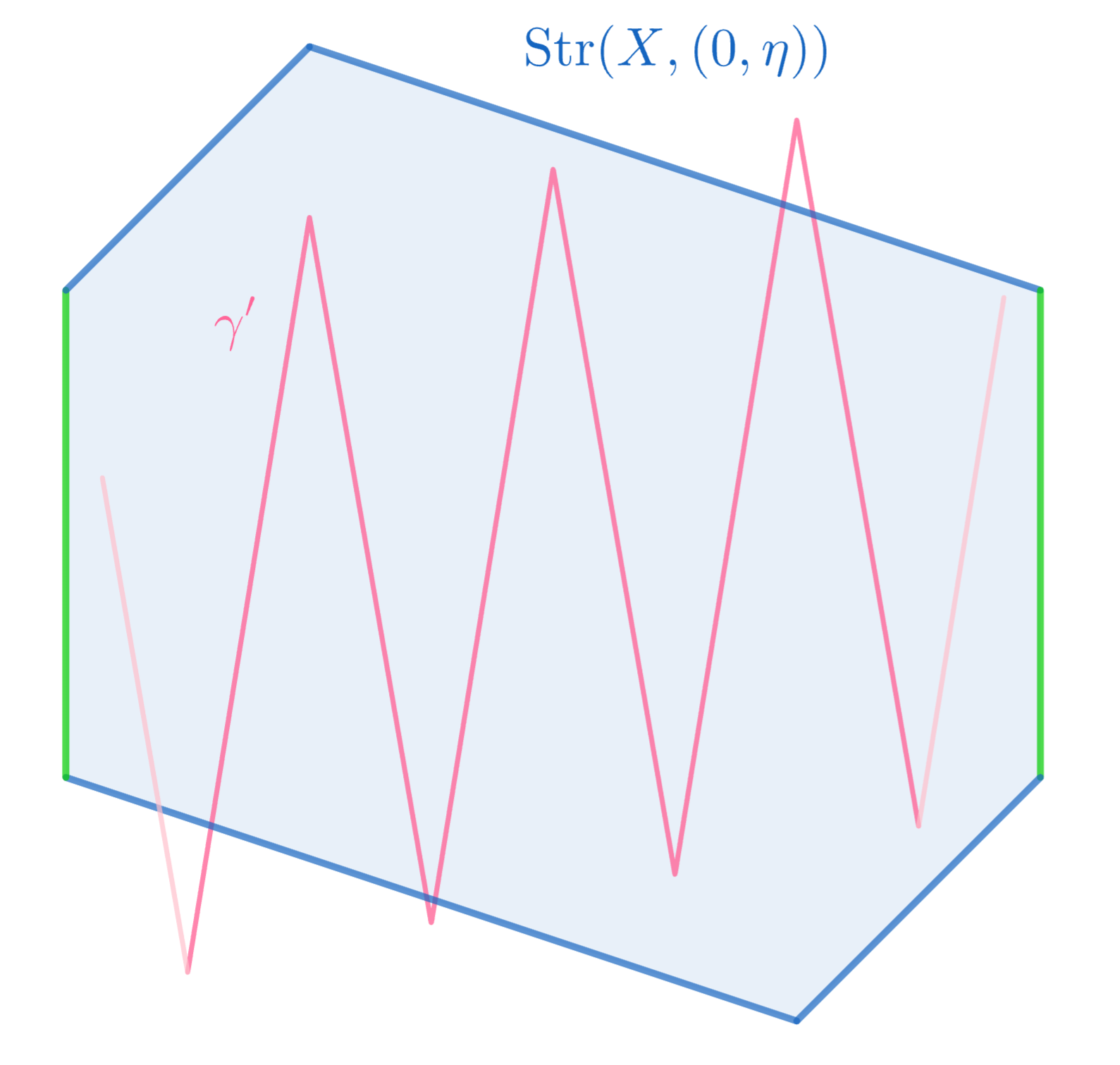}
        \caption{The image of $\gamma$ under $J_{\eta,\xi}$ and $\stre(X,(0,\eta))$}
        \label{fig:after-stretch}
    \end{subfigure}
        \caption{}
\end{figure}
    Write $I_i=[x_i,y_i]$ with end points $x_i=(a_i,b_i)$ and $y_i=(c_i,d_i)$. Denote the slope of $I_i$ by $m_i = (d_i-b_i)/(c_i-a_i)$, and let $\theta_i = \arctan(m_i) \in ]-\frac{\pi}{2},\frac{\pi}{2}[$ denote its positive angle. Let $\delta = \min \{\ell,\varepsilon'\} \times \cos\arctan(M)$. Then,
    \begin{equation*}
        \delta < \min_{i \in S} \{\min \{|I_i|,\varepsilon'\} \times \cos \theta_i\}.
    \end{equation*}
    Let $\xi_0$ be the following positive integer:
    \begin{equation*}
        \left \lceil\frac{1}{\delta}+\frac{M+m}{2|\eta|}\right\rceil
    \end{equation*}
    Fix $\xi \geq \xi_0$. For each $i$, the horizontal diameter of $I_i$, i.e. $|I_i|\cos \theta_i$, is greater than $\delta$ and by the choice of $\xi_0$, greater than $\frac{1}{\xi}$. Let $I'_i=[x'_i,y'_i]$ be the largest sub-segment of $I_i$ that projects into an interval of the form $[\frac{z^1_i}{2\xi},\frac{z^2_i}{2\xi}]$, with $z^1_i<z^2_i \in \m Z$. Write $x'_i=(a'_i,b'_i)$ and $y'_i=(c'_i,d'_i)$, with $a'_i=\frac{z^1_i}{2\xi}$ and $c'_i=\frac{z^2_i}{2\xi}$. By the choice of $I'_i$, we have $|a_i - a'_i|<\frac{1}{2\xi}$, and $|c_i - c'_i|<\frac{1}{2\xi}$. We know $\xi \geq \xi_0 \geq \delta\inv$, therefore
    \begin{equation}
    \label{eq:left-end-point}
        d(x_i,x'_i) = \frac{|a_i - a'_i|}{\cos \theta_i} < \frac{1}{2\xi \cos \theta_i} \leq \frac{\delta}{2\cos \theta_i} < \frac{\varepsilon'}{2}.
    \end{equation}
    In the same way,
    \begin{equation}
    \label{eq:right-end-point}
        d(y_i,y'_i) < \frac{\varepsilon'}{2}.
    \end{equation}
    Let $\gamma'$ be $\cup_{i \in S} J_{\eta,\xi}(I'_i)$. As the map $J_{\eta,\xi}$ is piecewise affine, $\gamma'$ is a union of segments. More precisely, each subsegment of $I'_i$ projecting to an interval of the form $[\frac{z}{2\xi},\frac{z+1}{2\xi}]$ in the horizontal direction will be mapped to a single segment by $J_{\eta,\xi}$. \par
    \begin{sloppypar}
        An arbitrary point in $\stre(X,(0,\eta))$ may be written as $w+(0,\tau)$, with $w \in X$ and $\tau \in [-|\eta|,|\eta|]$. By $\varepsilon$-density of $\gamma$ in $X$ and inequalities \eqref{eq:left-end-point} and \eqref{eq:right-end-point}, there exists $i \in S$ and $w' \in I'_i$ such that $d(w,w') < \varepsilon + \frac{\varepsilon'}{2}$. Write ${w' = (a',b')}$. Then $a'$ falls in an interval of the form $[\frac{z}{2\xi},\frac{z+1}{2\xi}] \subseteq [\frac{z^1_i}{2\xi},\frac{z^2_i}{2\xi}]$ where $z$ is an integer. Let $I_{w'}$ denote the subsegment of $I'_i$ that projects to $[\frac{z}{2\xi},\frac{z+1}{2\xi}]$ in the horizontal direction. As $\tau \in [-|\eta|,|\eta|]$, there exists $w''=(a'',b'') \in I_{w'}$ such that $J_{\eta,\xi}(w'')=w''+(0,\tau)$. As the horizontal distance between $w'$ and $w''$ is at most $\frac{1}{2\xi}$, then similar to the computations done before we see that $d(w',w'') \leq \frac{\varepsilon'}{2}$. We also had that $d(w,w') < \varepsilon + \frac{\varepsilon'}{2}$, so
    \end{sloppypar}
    \begin{equation*}
        d(w+(0,\tau),J_{\eta,\xi}(w''))=d(w+(0,\tau),w''+(0,\tau)) < \varepsilon+\varepsilon'.
    \end{equation*}
    Therefore $\gamma'=\cup_{i \in S} J_{\eta,\xi}(I'_i)$ is ($\varepsilon+\varepsilon'$)-dense in $\stre(X,(0,\eta))$. \par
    The map $J_{\eta,\xi}$ sends a segment of slope $m_i$ to segments of slopes $m_i \pm 4\eta\xi$. Given the choice of $\xi_0$,
    \begin{equation*}
        4|\eta|\xi - \max_{i \in S} |m_i| > 4|\eta|\xi - M > 2|\eta|\xi - M \geq m.
    \end{equation*}
    So all norms of slopes of segments of $J_{\eta,\xi}(\gamma)$ are greater than $m$. Therefore the same holds for the segments of $\gamma'$. The segments of $\gamma'$ have horizontal diameter ${1}/{2\xi}$, therefore their length is at least ${(4|\eta|\xi - M)}/{2\xi}=2|\eta|-{M}/{2\xi}$. By the choice of $\xi_0$, ${M}/{2\xi} < |\eta|$, therefore all segments in $\gamma'$ have length greater than $|\eta|$.
\end{proof}
The space of directions of non-degenerate vectors in $\rtwo$ can be identified with the unit circle $\sone$. Fix the Riemannian metric induced by the Euclidean metric of $\rtwo$ on $\sone$.
\begin{corollary}
\label{cor:induction-step}
    Let $X$ be a non-empty subset of $\rtwo$. Fix $\eta \in \m R^{*}$, and let $\varepsilon,\varepsilon',\ell,\delta,\delta' > 0$. Let $A \in \mrm{SL}(2,\m Z)$, and note $A(0,1)=v$. Let $\gamma$ be a union of segments in $\rtwo$ that is $\varepsilon$-dense in $X$, such that the distances of directions of segments of $\gamma$ from the direction of $v$ are greater than $\delta$ on $\sone$, and the lengths of the segments are greater than $\ell$. Then there exists $\xi_0=\xi_0(|\eta|,A,\ell,\varepsilon',\delta,\delta') \in \nstar$ such that for any $\xi \geq \xi_0$, there exists a subset $\gamma' \subseteq \conj{A}{J_{\eta,\xi}}(\gamma)$ with the following properties:
    \begin{itemize}
        \item $\gamma'$ is a union of segments,
        \item $\gamma'$ is ($\varepsilon+\varepsilon'$)-dense in $\stre(X,\eta\cdot v)$,
        \item the directions of segments of $\gamma'$ are $\delta'$-close to the direction of $v$ on $\sone$, and
        \item the lengths of the segments of $\gamma'$ are greater than $\opnorm{A\inv}\inv |\eta|$.
    \end{itemize}
\end{corollary}

\begin{proof}[Sketch of proof]
    The last lemma shows this result for $A=\mrm {Id}$. To prove this corollary, we apply the last lemma to $A\inv(\gamma)$ and $A\inv(X)$ with some appropriate parameters to obtain $\xi_0$. We then apply $A$ to the subset provided by the lemma to get the desired subset. \par
    
    This is how we find the parameters ($\varepsilon,\varepsilon',\ell,m,M$) for Lemma \ref{lem:spread-shape}. By linearity of $A$, the set $A\inv(\gamma)$ is a union of segments which have at least length $\opnorm{A}\inv \ell$. Since the distances of the directions of segments of $\gamma$ are $\delta$-away from that of $v$, one may compute the maximum slope of segments of $A\inv(\gamma)$ as a function of $A$ and $\delta$. Additionally, the set $A\inv(\gamma)$ is $\opnorm A \varepsilon$-dense in $A\inv(X)$. Take $\varepsilon'$ of the lemma to be $\opnorm{A}\inv \varepsilon'$. We may choose $m$ as a function of $A$ and $\delta'$ so as to insure that all vectors with slope more than $m$ will be mapped by $A$ to vectors whose directions are $\delta'$-close to that of $v$. \par
    
    Lemma \ref{lem:spread-shape} provides us with a subset of $J_{\eta,\xi}A\inv(\gamma)$ which is a union of segments with directions close enough to the vertical direction, and with length at least $|\eta|$. By applying $A$ to this subset we get a union of segments with directions $\delta'$-close to that of $v$, and with length at least $\opnorm{A\inv}\inv |\eta|$. As a corollary of the proof of the density requirement in Lemma \ref{lem:spread-shape}, for any point in $\stre(A\inv(X),(0,\eta))$, there exists a point in the subset provided by the lemma such that their difference can be written as $A\inv(w)+w'$ such that $|w|<\varepsilon$ and $|w'|<\opnorm{A}\inv \varepsilon'$. Therefore by applying $A$ to the subset provided by the lemma, we obtain a set that is $(\varepsilon+\varepsilon')$-dense in
    \begin{equation*}
        A(\stre(A\inv(X),(0,\eta)))=\stre(X,A(0,\eta))=\stre(X,\eta\cdot v). \qedhere
    \end{equation*}
\end{proof}
\begin{lemma}
\label{lem:induction-step}
    Let $X,Y$ be two non-empty subsets of $\rtwo$. Fix $\eta \in \m R^{*}$, $\xi \in \nstar$, and $\varepsilon > 0$. Let $A \in \mrm{SL}(2,\m Z)$. If $Y \subseteq B_{\varepsilon}(X)$, then
    \begin{equation*}
        \conj{A}{J_{\eta,\xi}}(Y) \subseteq B_{\varepsilon}(\stre(X,A(0,\eta))).
    \end{equation*}
\end{lemma}

\begin{proof}
    Let $y$ be a point in $Y$. There exists $x \in X$ such that $d(x,y)<\varepsilon$, and thus $d(x+A(0,\tau),y+A(0,\tau))<\varepsilon$. We may write 
    \begin{equation*}
        \conj{A}{J_{\eta,\xi}}(y)=A(A\inv(y)+(0,\tau))=y+A(0,\tau),
    \end{equation*}
    where $\tau$ is in $[-|\eta|,|\eta|]$. As we know that $d(y+A(0,\tau),x+A(0,\tau))<\varepsilon$, and $x+A(0,\tau) \in \stre(X,A(0,\eta))$,
    \begin{equation*}
        \conj{A}{J_{\eta,\xi}}(y) \in B_{\varepsilon}(\stre(X,A(0,\eta))). \qedhere
    \end{equation*}
\end{proof}

\begin{lemma}
    \label{lem:spreading-to-parallelogon}
    Let $r >0$ and $\Gamma \in \shapes$ the homothety type of a rational zongon. Then, there exist $\xi_0 \in \nstar$ such that for any $\xi \geq \xi_0$, there exists $h \in \mrm{Diff}_0^{\infty}(\torus)$ such that for any $a \in {1}/{(2\xi)} + (1/\xi) \m Z$, and $b \in \m R$, $\conj{\Tilde{h}}{\Tilde{R}_{(a,b)}}(D)$ is an $r$-large approximate of $\Gamma$.
\end{lemma}

\begin{proof}
    To show this result, we construct a piecewise affine torus homeomorphism with the desired property. Then we may find the desired diffeomorphism $h$ by a small perturbation of this homeomorphism. We start by introducing the lift of the desired piecewise affine torus homeomorphism $F:\rtwo \to \rtwo$. This map is going to be a product of maps which spread domains in different directions. We then introduce a sequence of domains $(D_i)$ which terminates with $\conj{F}{\Tilde{R}_{(a,b)}}(D)$, and a sequence of zonogons $(K_i)$ terminating with homothety type $\Gamma$. These sequences approximate each other with respect to the Hausdorff distance. We conclude by smoothing out the map $F$.
    \begin{itemize}[leftmargin=*]
        \item \underline{Introducing F:} Let $\zon(v_0,\dots,v_{l-1})$ be a rational zonogon with diameter greater than $6r+10$, representing the homothety type $\Gamma$. Additionally, take $v_l=(0,1)$. For each $0 \leq i \leq l$, there exists $A_i \in \mrm{SL}(2,\m Z)$ and $\eta_i \in \m R^*$ such that $A_i(0,\eta_i)={v_i}/{2}$. In particular, $A_l = \mrm{Id}$ and $\eta_l=1/2$. Fix
        \begin{equation}
        \label{eq:lambda}
            \lambda = \min_{0 \leq i \leq l} \set {\opnorm{A_i\inv}\inv|\eta_i|}.
        \end{equation}
        Let $\varepsilon'={1}/{2l}$, and let $\delta=\delta'$ to be half the minimal distance between the directions of $v_1,\dots,v_l$ on $\sone$. Then by applying Corollary \ref{cor:induction-step} we may define $\xi_0 \coloneqq \max_{0 \leq i \leq l} \xi_0(|\eta_i|,A_i,\lambda,\varepsilon',\delta,\delta')$. Let $\xi \geq \xi_0$. Define $F:\rtwo \to \rtwo$ as follows:
        \begin{equation*}
            F = (\conj{A_0}{J_{\eta_0,\xi}})\dots(\conj{A_l}{J_{\eta_l,\xi}})
        \end{equation*}
        As $A_l=\mrm{Id}$ and $\eta_l=1/2$, $F=(\conj{A_0}{J_{\eta_0,\xi}})\dots(\conj{A_{l-1}}{J_{\eta_{l-1},\xi}})J_{1/2,\xi}$. Let $a \in {1}/{(2\xi)} + \m Z$ and $b \in \m R$. We have $J_{1/2,\xi}\Tilde{R}_{(a,b)}J_{-1/2,\xi} = \Tilde{R}_{(a,b)}J_{-1,\xi}$. Therefore,
        \begin{multline*}
            \conj{F}{\Tilde{R}_{(a,b)}}=(\conj{A_0}{J_{\eta_0,\xi}})\dots(\conj{A_{l-1}}{J_{\eta_{l-1},\xi}})\Tilde{R}_{(a,b)}J_{-1,\xi}\\(\conj{A_{l-1}}{J_{-\eta_{l-1},\xi}})\dots(\conj{A_0}{J_{-\eta_0,\xi}}).
        \end{multline*}
        For all $0 \leq i \leq l-1$, denote $\conj{A_i}{J_{-\eta_i,\xi}}$ as $F_i$. Let $F_l=\Tilde{R}_{(a,b)}J_{-1,\xi}$. For all $l+1 \leq i \leq 2l$, let $F_i$ denote $\conj{A_{2l-i}}{J_{\eta_{(2l-i)},\xi}}=F_{2l-i}\inv$. Hence $$\conj{F}{\Tilde{R}_{(a,b)}} = F_{2l} \circ \dots \circ F_0.$$
        If $v_{l-1}$ is in the vertical direction, $A_{l-1}=\mrm{Id}$ and thus
        \begin{equation*}
            \conj{F}{\Tilde{R}_{(a,b)}}=F_{2l}\dots F_{l+2}\Tilde{R}_{(a,b)}J_{-2\eta_{l-1}-1,\xi}F_{l-2}\dots F_0.
        \end{equation*}
        For the rest of the proof, we assume $v_{l-1}$ is not vertical. The vertical case is similar.

        \item \underline{Introducing $D_i$ and $K_i$:} Let $D_0=D$ and $K_0=\set{(1/2,1/2)}$. We construct a sequence of fundamental domains $(D_i)_{0\leq i \leq 2l+1}$ and a sequence of zonogons $(K_i)_{0\leq i \leq 2l+1}$ recursively. For any $0\leq i \leq 2l$, let $D_{i+1}=F_i(D_i)$. Therefore $$D_{2l+1}=\conj{F}{\Tilde{R}_{(a,b)}}(D).$$
        For any $0\leq i \leq l-1$, let $K_{i+1}=\stre(K_i,v_{i}/2)=(1/2,1/2)+\zon(v_0/2,\dots,v_{i}/2)$. Take $K_{l+1}=\Tilde{R}_{(a,b)}(K_l)=(a+1/2,b+1/2)+\zon(v_0/2,\dots,v_{l-1}/2)$. Finally, for any $l+1\leq i \leq 2l$, let $K_{i+1}=\stre(K_i,v_{2l-i}/2)$. Hence we have $$K_{2l+1}=(a+1/2,b+1/2)+\zon(v_0,\dots,v_{l-1}) \in \Gamma.$$ \par
        We wish to show that $\conj{F}{\Tilde{R}_{(a,b)}}(D)$ is an $(r+1)$-large approximate of $\Gamma$. In other words, we would like to prove that $D_{2l+1}$ is an $(r+1)$-large approximate of $[K_{2l+1}]$. To do this, we need to show that the Hausdorff distance between $D_{2l+1}$ and $K_{2l+1}$ is sufficiently small.

        \item \underline{Proving $D_{2l+1} \subseteq B_2(K_{2l+1})$}: Let us show by induction that for $0 \leq i \leq l$, $D_i \subseteq B_1(K_i)$. By choice of $D_0$ and $K_0$, we have $D_0 \subseteq B_1(K_0)$. Suppose that $D_i \subseteq B_1(K_i)$ for some $0 \leq i \leq l-1$. By Lemma \ref{lem:induction-step},
        \begin{equation*}
            \conj{A_i}{J_{-\eta_i,\xi}}(D_i) \subseteq B_1(\stre(K_i,A_i(0,-\eta_i))).
        \end{equation*}
        Therefore we have $F_i(D_i) \subseteq B_1(\stre(K_i,v_i/2))$, so $D_{i+1}\subseteq B_1(K_{i+1})$. Again by Lemma \ref{lem:induction-step},
        \begin{equation*}
            J_{-1,\xi}(D_l)\subseteq B_1(\stre(K_l,(0,1))) \subseteq B_2(K_l).
        \end{equation*}
        Therefore $D_{l+1}=\Tilde{R}_{(a,b)}J_{-1,\xi}(D_l)\subseteq B_2(\Tilde{R}_{(a,b)}K_l)=B_2(K_{l+1})$. In a similar way using induction we can prove that for all $l+1 \leq i \leq 2l+1$ we have $D_i \subseteq B_2(K_i)$.

        \item \underline{Proving $K_{2l+1} \subseteq D_2(K_{2l+1})$}: Let $v_{-1}$ denote any vector whose direction is $2\delta$-away from that of $v_0$. We construct a sequence $(\gamma_i)_{0 \leq i \leq 2l+1}$ of subsets of $\rtwo$, such that for each $i$,
        \begin{itemize}
            \item $\gamma_i \subseteq D_i$,
            \item $\gamma_i$ is a union of segments,
            \item $\gamma_i$ is $(i+1)/(2l)$-dense in $K_i$,
            \item The directions of segments of $\gamma_i$ are $\delta$-close to the direction of $v_{i-1}$ on $\sone$, and
            \item the lengths of segments of $\gamma_i$ are greater than $\lambda$ (see equation \eqref{eq:lambda}).
        \end{itemize}
        Let $\gamma_0$ be the maximal segment of direction $v_{-1}$ passing through $K_0$ included in $D=D_0$. Its length will be at least $1 \geq \lambda$. The segment $\gamma_0$ is $1/(2l)-dense$ in $K_0$, since $K_0 \subseteq \gamma_0$. Hence $\gamma_0$ has the desired property. Suppose we have $\gamma_i$ for some $0\leq i \leq l-1$ with the above properties. As the directions of its segments are $\delta$-close to the direction of $v_{i-1}$, they are $\delta$-away from the direction of $v_i$. Therefore by Corollary \ref{cor:induction-step}, there exists a union of segments $\gamma_{i+1}$ such that
        \begin{equation*}
            \gamma_{i+1} \subseteq \conj{A_i}{J_{-\eta_i,\xi}}(\gamma_i) = F_i(\gamma_i) \subseteq F_i(D_i) = D_{i+1},
        \end{equation*}
        and such that $\gamma_{i+1}$ is $(i+2)/2l$-dense in $K_{i+1}=\stre(K_i,v_{i}/2)$. These segments have at least length $\opnorm{A_i\inv}\inv|\eta_i| \geq \lambda$, and their directions are $\delta$-close to $v_{i}$. Similarly, there exists a union of segments $\gamma_{l+1}$ such that
        \begin{equation*}
            \gamma_{l+1} \subseteq \Tilde{R}_{(a,b)}J_{-1,\xi}(\gamma_l) \subseteq D_{l+1},
        \end{equation*}
        and $\gamma_{l+a}$ is $(l+2)/2l$-dense in $(a,b)+\stre(K_l,(0,1))$. But $K_{l+1}$ is a subset of $(a,b)+\stre(K_l,(0,1))$, so $\gamma_{l+1}$ is $(l+2)/2l$-dense in $K_{l+1}$. Like before, by applying Corollary \ref{cor:induction-step} for $\gamma_i$, $F_i$, and $K_i$ where $l+1 \leq i \leq 2l$, we may finish the construction of the sequence $(\gamma_i)$ with the above properties. In particular, $\gamma_{2l+1}$ is included in $D_{2l+1}$ and is $(2l+2)/2l$-dense and hence $2$-dense in $K_{2l+1}$. Therefore $K_{2l+1} \subseteq B_2(D_{2l+1})$.

        \item \underline{Conclusion:} We have already shown that $D_{2l+1} \subseteq B_2(K_{2l+1})$, so the Hausdorff distance between $D_{2l+1}$ and $K_{2l+1}$ is less than $2$. Given that $\diam(K_{2l+1})$ equals $\diam(\zon(v_0,\dots,v_{l-1}))$, we have
        \begin{equation*}
            \diam(D_{2l+1}) \geq \diam(K_{2l+1}) - 4 > 6r+6 > r+1.
        \end{equation*}
        Consequently, the Hausdorff distance between $D_{2l+1}/\diam(D_{2l+1})$ and a well-chosen translation of $K_{2l+1}/\diam(K_{2l+1})$ will be less than $6/(\diam(K_{2l+1})-4)$ and hence less than $1/(r+1)$. Therefore
        \begin{equation*}
            \isla{\conj{F}{\Tilde{R}_{(a,b)}}(D)}{r+1}{\Gamma}.
        \end{equation*}
        The plane homeomorphism $F$ quotients down to a torus homeomorphism $f$. According to Lemma \ref{lem:perturbation}, for any torus diffeomorphism $h$ that has sufficiently small $C^0$ distance from $f$,
        \begin{equation*}
            \isla{\conj{\Tilde{h}}{\Tilde{R}_{(a,b)}}(D)}{r}{\Gamma}.
        \end{equation*}
        Each of the maps $F_i$ is accumulated by plane diffeomorphisms commuting with $\ztwo$. We can obtain these diffeomorphisms by considering smooth $1/\xi$-periodic approximations of the function $\phi_{\xi}$. In this way we can find a smooth torus diffeomorphism that is sufficiently close to $f$. \qedhere
    \end{itemize}
\end{proof}

\begin{proof}[Proof of Lemma \ref{lem:two}]
\begin{sloppypar}
    By Lemma \ref{lem:C}, for a large enough $j$, we have $\conj{C_q}{P_{j}^{C_q\inv(\Gamma)}} \subseteq P_{\ell}^{\Gamma}$. Let $\xi_0$ and $h_0$ associated to it be as given by Lemma \ref{lem:spreading-to-parallelogon} applied to $r=j$ and the rational zonogon $C_q\inv(\Gamma)$. We may find a sequence $(\alpha_i)$ converging to $C_q\inv(p/q,0)=(p,0)$, such that $\alpha_i = ((2s_i+1)/(2t_i \xi_0),0)$, with $s_i 
    \in \m Z$ and $t_i \in \m Z^*$. Hence $t_i\alpha_i = (s_i/\xi_0 + {1}/{(2\xi_0)},0)$. Therefore for each $i$, $\conj{\Tilde{h}_0}{\Tilde{R}_{\alpha_i}^{t_i}}(D)$ is a $j$-large approximate of ${C_q\inv(\Gamma)}$, so $\conj{h_0}{R_{\alpha_i}} \in P_{j}^{C_q\inv(\Gamma)}$.
\end{sloppypar}
    Define $h \coloneqq \conj{C_q}{h_0}$ (using Lemma \ref{lem:conj-by-C}). This torus diffeomorphism commutes with $R_{(1/q,0)}$. For each $i$, by Lemma \ref{lem:C},
    \begin{equation*}
        \conj{h}{R_{C_q(\alpha_i)}} = \conj{C_q}{\conj{h_0}{R_{\alpha_i}}} \in \conj{C_q}{P_{j}^{C_q\inv(\Gamma)}} \subseteq P_{\ell}^{\Gamma}.
    \end{equation*}
    Hence the homeomorphism $h$ and the sequence of vectors $(\theta_i)_i=(C_q(\alpha_i))_i$ satisfy the conditions required by the lemma.
\end{proof}

\bibliographystyle{alpha}
\bibliography{sources}

\end{document}